\documentclass[a4paper]{amsart}
\usepackage{amsmath,amsthm,amsfonts,amssymb,amscd,mathrsfs}
\usepackage[left=3.5cm,right=3.5cm,top=4.0cm,bottom=4.0cm]{geometry}
\usepackage{graphicx}
\usepackage{color}
\usepackage[all]{xy}
\usepackage{enumerate}
\usepackage{array}
\usepackage{dsfont}

\usepackage{color}

\theoremstyle{plain}

\newtheorem{theorem}{Theorem}[section]

\newtheorem{lemma}[theorem]{Lemma}

\newtheorem{remark}[theorem]{Remark}
\newtheorem{pro}[theorem]{Proposition}

\newtheorem*{step1}{Step 1}
\newtheorem*{step2}{Step 2}
\newtheorem*{step3}{Step 3}
\newtheorem*{step4}{Step 4}

\begin{document}
\title{Surfaces with $\chi=5, K^{2}=9$ and a canonical involution}
\author{Zhiming Lin}

\let\thefootnote\relax

\maketitle



\bigskip

\section{introduction}

Complex nonsingular minimal projective surfaces of general type with $K^2=2\chi-n$ have been classified for various $n$'s ($n\leq6$ by Noether's inequality): Horikawa \cite{Horikawa 1} \cite{Horikawa 3} for $4\leq n\leq6$; Liu \cite{Liu} for $n=3$ and $\chi\geq6$; Bauer \cite{Bauer 1} for $n=3$ and $\chi=5$; Marti-Sanchez \cite{Marti Sanchez} for $n=2$ and $\chi\geq6$; Bauer and Pignatelli \cite{Bauer 2} for $n=2$, $\chi=5$ and with a canonical involution; Catanese, Liu and Pignatelli \cite{Catanese 2} for $n=2$, $\chi=5$ and with an even canonical divisor;  Werner \cite{Werner} for $n=1$ and $\chi\geq 7$; Murakami \cite{Murakami 1} \cite{Murakami 2} for $n=1$, $\chi=4$ and with a non-trivial torsion.

In this paper we consider the case $n=1$ and  $\chi=5$. We follow the arguments of Bauer and Pignatelli \cite{Bauer 2} to classify these surfaces under the assumption that the canonical map factors through an involution. Our main results are the following (see Proposition \ref{dimension} and Theorem \ref{component} for more precise statement).

\begin{theorem}\label{main}
Let $S$ be a complex nonsingular minimal projective surface of general type with $\chi=5$ and $K^2=9$. We assume that the canonical map of $S$ factors through an involution. Then the moduli space of such surfaces consists of six families, whose dimensions are $28$, $27$, $33$, $32$, $31$ and $32$ respectively. Among them, the family of surfaces having a genus $2$ fibration forms an irreducible component of $\mathcal{M}_{5, 9}$, the moduli space of surfaces with $\chi=5$ and $K^2=9$.
\end{theorem}

The paper is organized as follows.

We recall some tools in Section 2, then we show that the canonical involution $i$ has 1, 3 or 5 isolated fixed points in Section 3. Sections 4, 5 and 6 are devoted to classifying surface case by case according to the number of the fixed points of $i$.

In Section 7, we study the moduli space of these minimal surfaces. By Kuranishi's theorem each irreducible component of the moduli space $\mathcal{M}_{5, 9}$ has dimension at least 32. In addition, the general point of the irreducible component, in which the family of dimension 30 or 27 is contained, is a surface without a canonical involution.

\bigskip

\textbf{Acknowledgement.} I am heavily indebted to my supervisor Professor Jinxing Cai for his patient guidance and helping me out of the difficulties during the hard time of this paper. My special thanks go to Wenfei Liu and Lei Zhang for sharing many of their ideas and revisions. I also thank Yi Gu, Jingshan Chen and Songbo Ling for numerous useful discussions.

\bigskip

\textbf{Conventions.} We work over the complex field $\mathbb{C}$. By surface we mean a smooth projective surface. We do not distinguish between line bundles and divisors on a smooth surface, and use the additive and the multiplicative notation interchangeably. We write $\equiv$ for linear equivalence and $\sim$ for numerical equivalence. Unless specifically stated, we use the standard definition in \cite{BPV} of the usual notations.


\bigskip

\section{Preliminaries}

We recall some basic results about the involution on a surface.

Let $S$ be a minimal surface which has an involution $i$. The fixed locus of $i$ is the union of a smooth curve $R$ (possibly empty) and isolated points $p_{1}, \ldots, p_{\tau}$. Let $\pi':S\rightarrow S/i$ be the quotient map. Then the surface $S/i$ is normal and $\pi'(p_{1}), \ldots,\pi'(p_{\tau})$ are ordinary double points, which are the only singularities of $S/i$. Resolving these singularities we get a commutative diagram
$$\xymatrix{
    \widehat{S} \ar[r]^{\epsilon}\ar[d]_{\pi} & S \ar[d]_{\pi'} \\
    \widehat{T} \ar[r]^{} & S/i
    }$$
where $\epsilon$ is the blowing-up of S at $p_{1}, \ldots, p_{\tau}$, $\widehat{T}=\widehat{S}/\widehat{i}$ is a smooth surface and $\widehat{i}$ is the involution on $\widehat{S}$ induced by $i$. Let $E_{i}\doteq \epsilon^{-1}(p_{i})$ be the ($-1$)-curves on $\widehat{S}$, so $A_{i}\doteq\pi(E_{i})$ are the ($-2$)-curves on $\widehat{T}$. Put $\widehat{\delta}\doteq\frac{1}{2}(\pi(R)+A)$ and $\overline{\delta}\doteq\widehat{\delta}-\frac{1}{2}A=\frac{1}{2}\pi(R)$, where $A\doteq\sum_{i=1}^{\tau}A_{i}$. Sometimes, when we study the birational morphism, we do not distinguish the curve from the total inverse image of it for convenience (i.e. the curve $R$ on $S$ and $\epsilon^{*}R$ on $\widehat{S}$).

\begin{lemma}\label{main lemma}
With the notation as above, we have
\begin{enumerate}[1)]
\item $K_{\widehat{T}}+\overline{\delta}$ is big and nef;
\item $K^{2}_{S}-\tau=K^{2}_{\widehat{S}}=2(K_{\widehat{T}}+\widehat{\delta})^{2}$ and $K^{2}_{S}=2(K_{\widehat{T}}+\overline{\delta})^{2}$;
\item $h^{i}(2K_{\widehat{T}}+\widehat{\delta})=0(i>0)$ and $h^{0}(2K_{\widehat{T}}+\widehat{\delta})=\chi(\mathcal{O}_{\widehat{T}})+ \frac{1}{2}(2K_{\widehat{T}}+\widehat{\delta}) (K_{\widehat{T}}+\widehat{\delta})$;
\item $\chi(\mathcal{O}_{S})=2\chi(\mathcal{O}_{\widehat{T}})+\frac{1}{2}(K_{\widehat{T}}+ \widehat{\delta})\widehat{\delta}$;
\item $\chi(\mathcal{O}_{\widehat{T}})=\frac{1}{2}\chi(\mathcal{O}_{S})-\frac{1}{8}(K_{S}R-\tau)$;
\item $0\leq \tau =K_{S}^{2}+6\chi(\mathcal{O}_{\widehat{T}})-2\chi(\mathcal{O}_{S})- 2h^{0}(2K_{\widehat{T}}+\widehat{\delta})$.
\end{enumerate}
\end{lemma}
\begin{proof}
1), 2) and 3) are from \cite{Calabri} Proposition 3.1; 4) and 5) are from \cite{Bauer 2} Lemma 1.1; 6) is an immediate consequence of 2), 3) and 4).
\end{proof}

A canonical involution on $S$ is an involution which factors through the canonical map. With the help of the following lemma, we only consider the surface having a canonical involution in this paper.

\begin{lemma}\label{canonical involution}
With the notation as above, if $i$ is a canonical involution on $S$, then either $p_{g}(\widehat{T})=0$ or $p_{g}(\widehat{T})=p_{g}(S)$. In the former case every isolated fixed point of $i$ is a base point of $|K_{S}|$, and in the later case $R$ is contained in the fixed part of $|K_{S}|$.
\end{lemma}
\begin{proof}
See \cite{Bauer 2} Lemma 1.3.
\end{proof}

\bigskip

To finish the classification, we recall the contraction theorem using the notation in \cite{Kollar}.

\begin{lemma}\label{Contraction Theorem}
\textbf{(Contraction Theorem)} Let $P$ be a smooth surface. For each extremal ray $l$ in the half space $\overline{NE(P)}_{K_{P}<0}$, there exists the associated extremal contraction $\varphi_{l}:P\rightarrow W$. Moreover, $\varphi_{l}$ is one of the following type
\begin{enumerate}[1)]
\item W is a smooth surface and $P$ is obtained from W by blowing-up a point; $\rho(W)=\rho(P)-1$.
\item W is a smooth curve and $P$ is a minimal ruled surface over W; $\rho(P)=2$.
\item W is a point, $\rho(P)=1$ and $-K_{P}$ is ample; In fact $P\cong \mathds{P}^{2}$.
\end{enumerate}
\end{lemma}
\begin{proof}
See \cite{Kollar} theorem 3.7.
\end{proof}

\begin{lemma}\label{extremal ray}
Let $P$ be a smooth surface and $\delta \in Pic(P)\otimes \mathbb{Q}$. If
$$r=max\{t\in \mathbb{R}|\delta+tK_{P}\text{ is nef }\}<+\infty,$$
then there exists $l\in \overline{NE(P)} \verb|\| \{ 0\}$ satisfying
\begin{enumerate}[1)]
\item $(rK_{P}+\delta)l=0$;
\item $K_{P}\cdot l<0$;
\item $\mathbb{R}_{\geq0}[l]$ is an extremal ray. (We don't distinguish $\mathbb{R}_{\geq0}[l]$ from $l$ below.)
\end{enumerate}
\end{lemma}
\begin{proof}
Let $M_{r}\doteq \{l|(r K_{S}+\delta)l=0\}$ and $NE(S)^{0}$ be the interior of $\overline{NE(S)}$. With the increase of $r$, the intersection of $M_{r}$ and $NE(S)^{0}$ will become nonempty from an empty set. If $r$ is the maximal number such that $rK_{S}+\delta$ is nef, then there exists $l' \in \overline{NE(S)} \verb|\| \{ 0\}$ with $(rK_{S}+\delta)l'=0$ and $K_{S}\cdot l'<0$. By cone theorem, we have $l'=l''+\Sigma a_{n}R_{n}$ with $K_{S}l''\geq 0$ and $K_{S}R_{n}<0$ and $a_{n}\geq0$ for any $n$. Since $K_{S}\cdot l'<0$, there exists an integer $i$ such that $a_{i}\neq0$. Let $l=R_{i}$. Since $(rK_{S}+\delta)l'=0$ and $rK_{S}+\delta$ is nef, $l$ satisfies the three conditions.
\end{proof}

\bigskip

To prove the existence of each family, we only need to show that the smooth branch curve $2\widehat{\delta}$ on $\widehat{T}$ is really existent. By Bertini's theorem, it is enough to prove the base point free of $|2\widehat{\delta}|$.

\begin{lemma}\label{base point free}
Let $P$ be a smooth surface and $L$ is a nef divisor on $P$. If we assume $L^{2}\geq 5$ and $L \equiv (2a+1)K_{P}+2bD$ where $a,b\in \mathbb{Z}$ and $D \in \mathrm{Pic}(P)$, then $|K_{P}+L|$ base point free.
\end{lemma}
\begin{proof}
It is an immediate consequence of the Reider's method (See \cite{BPV} IV 11.4).
\end{proof}

\begin{lemma}\label{base point free 1}
Let $P$ be a smooth surface, $E$ a (-1)-curve on $P$, $|\Omega|$ base point free on $P$ and $\Omega E>0$. If we assume that $\Omega\equiv K_{P}+L+D$, where $D$ is a normal crossing divisor and $L$ is big and nef, then $|\Omega+E|$ is also base point free.
\end{lemma}
\begin{proof}
Since $\Omega E>0$ and $\Omega\equiv K_{P}+L+D$, E is not a base part of $|\Omega+E|$ by vanishing theorem and Riemann-Roch theorem. So $H^{0}(P,\Omega)$ is a proper subspace of $H^{0}(P,\Omega+E)$. For any point $p\in E$, if $p$ is a base point of $|\Omega+E|$, then it must be a base point of $|\Omega|$.
\end{proof}

If $W$ is a smooth surface with a big and nef divisor $-K_{W}$, then we call $W$ a weak del Pezzo surface. For example, $\mathds{P}^{2}$ and the Hirzebruch surfaces $\mathds{F}_{n}$ are weak del Pezzo surfaces if $0\leq n\leq 2$.

\begin{lemma}\label{nef 1}
Let $W$ be a weak del Pezzo surface and $\alpha : W'\rightarrow W$ be the blowing-up with center at a point $p$ not lying on any (-2)-curve.
\begin{enumerate}[1)]
\item If $K_{W'}^{2}\geq 1$, then $W'$ is also a weak del Pezzo surface. In particular, $-K_{W'}$ is nef.
\item If $K_{W'}^{2}= 0$, then $-K_{W'}$ is also nef.
\end{enumerate}
\end{lemma}
\begin{proof}
See \cite{Dolgachev} Proposition 8.1.16.
\end{proof}

\bigskip

\section{Numerical Classification}

\begin{lemma}
If $S$ is a minimal surface with $\chi(\mathcal{O}_{S})=5$ and $ K^{2}_{S}=9$, then
\begin{enumerate}[1)]
\item  $q(S)=0$ and $p_{g}(S)=4$;
\item The canonical map of $S$ is not composed with a pencil.
\end{enumerate}
\end{lemma}
\begin{proof}
1) are from \cite{Bombieri} Lemma 14; 2) are from \cite{Zucconi} Theorem 4.1.
\end{proof}

\begin{pro}
Let $S$ be a minimal surface with $\chi(\mathcal{O}_{S})=5, K^{2}_{S}=9$ having a canonical involution $i$, $\tau$ the number of the isolated fixed points of $i$, $\epsilon: \widehat{S} \rightarrow S$ the blowing-up at the isolated fixed points, $\widehat{i}$ an involution on $\widehat{S}$ induced by $i$ and $\widehat{T}=\widehat{S}/\widehat{i}$. Then we have $p_{g}(\widehat{T})=0$ and $\tau \in \{1,3,5\}$.
\end{pro}
\begin{proof}
We have that $p_{g}(\widehat{T})=0$ or $p_{g}(\widehat{T})=4$ by Lemma \ref{canonical involution} and Lemma 3.1.

If $p_{g}(\widehat{T})=4$, then we consider the canonical model $\overline{T}$ of $\widehat{T}$. Since $K^2_{\overline{T}}\leq 4$ and $|K_{\overline{T}}|$ is not composed with a pencil by Lemma 3.1, we have $\mathrm{deg}\phi_{K_{\overline{T}}}=2$ and $\mathrm{deg}\phi_{K_{\overline{T}}}(\overline{T})=2$. Since $K_{S}$ is 2-connected and $8=\mathrm{deg}\phi_{K_{S}}\cdot \mathrm{deg} \phi_{K_{S}}(S) < K^{2}_{S}=9$, $|K_{S}|$ has no a base part and has a simple base point. Now $R=\emptyset$ by Lemma \ref{canonical involution} and $\tau=20$ by Lemma \ref{main lemma} 5). We have $2(K_{\widehat{T}}+\widehat{\delta})^{2}=-11$ by Lemma \ref{main lemma} 2). This contradicts that $K_{\widehat{T}}+\widehat{\delta}$ is a Cartier divisor.

If $p_{g}(\widehat{T})=0$, then $0\leq \tau =5-2h^{0}(2K_{\widehat{T}}+\widehat{\delta})$ by Lemma \ref{main lemma} 6), which gives the result $\tau \in \{1,3,5\}$.
\end{proof}

As a corollary of the Proposition 3.2, we note that the minimal surface with $\chi(\mathcal{O}_{S})=5$ and $K^{2}_{S}=9$ has at most one canonical involution.

\bigskip

\section{the case $\tau=1$}

We consider the case $\tau=1$. Let $p$ be the unique isolated fixed point of the involution $i$, $E\doteq\epsilon^{-1}(p)$ and $A\doteq\pi(E)$.

With the help of contraction theorem, the property of $\widehat{T}$ is described by the following four steps. Indeed, we find a smooth surface $P_{s}$ and a birational morphism $f :\widehat{T} \rightarrow P_{s}$ such that $f_{*}(2K_{\widehat{T}}+\overline{\delta})$ is $\mathbb{Q}$-effective and nef. There are at most two choices of surface $P_{s}$ and morphism $f$.

\bigskip

\begin{step1}
$K_{\widehat{T}}+\overline{\delta}$ is nef and $2K_{\widehat{T}}+\overline{\delta}$ is $\mathbb{Q}$-effective.
\end{step1}
\begin{proof}
Since $\pi ^{*}(K_{\widehat{T}}+ \overline{\delta})\equiv \epsilon^{*}K_{S}$, $K_{\widehat{T}}+ \overline{\delta}$ is nef. By Lemma \ref{main lemma} 6) and $\tau=1$, we have $h^{0}(\mathcal{O}_{\widehat{T}}(2K_{\widehat{T}}+\widehat{\delta}))>0$. By $A(4K_{\widehat{T}}+2\widehat{\delta})=-2<0$, $A$ is contained in the fixed part of $|4K_{\widehat{T}}+2\widehat{\delta}|$. So we have $h^{0}(\mathcal{O}_{\widehat{T}}(2(2K_{\widehat{T}}+\overline{\delta})))= h^{0}(\mathcal{O}_{\widehat{T}}(4K_{\widehat{T}}+2\widehat{\delta}))>0$.
\end{proof}

\begin{step2}
There exists a smooth surface $P$ and a birational morphism $\alpha:\widehat{T}\rightarrow P$ contracting some (-1)-curves $l_{i}$ with $(K_{\widehat{T}}+\overline{\delta})l_{i}=0$. Let $\overline{\delta}_{P}\doteq \alpha_{*}(\overline{\delta})$. Then
\begin{enumerate}[1)]
\item $\overline{\delta}_{P}^{2}=\frac{17}{2}+K_{P}^{2}$ and $K_{P}\overline{\delta}_{P}=-2-K_{P}^{2}$;
\item $\frac{3}{2}K_{P}+\overline{\delta}_{P}$ is nef;
\item $K_{P}^{2}\in \{-3,-2,-1,0\}$.
\end{enumerate}
\end{step2}
\begin{proof}
In this case, we have the numerical conditions $K^{2}_{S}=9,\tau=1,\chi(\mathcal{O}_{S})=5, \chi(\mathcal{O}_{\widehat{T}})=1$, $A^{2}=-2$ and $K_{\widehat{T}}A=0$. Using the two equations in Lemma \ref{main lemma} 2) and 4), we can get $\overline{\delta}^{2}=\frac{17}{2}+K_{\widehat{T}}^{2}$ and $K_{\widehat{T}}\overline{\delta}=-2-K_{\widehat{T}}^{2}$.

Let $\lambda \in \mathbb{Q}\cup \{\infty \}$ be the maximal number such that $\lambda K_{\widehat{T}}+\overline{\delta}$ is nef. By Step 1, we have $\lambda \geq 1$. If $\lambda = 1$, then there exists an extremal ray $l$ with $(K_{\widehat{T}}+\overline{\delta})l=0$ and $K_{\widehat{T}} l<0$ by Lemma \ref{extremal ray}. Clearly, $l$ is a ($-1$)-curve by index theorem and  $(K_{\widehat{T}}+\overline{\delta})^{2}=\frac{9}{2}>0$. So we have a birational morphism $\alpha':\widehat{T}\rightarrow P'$ contracting $l$, where $P'$ is a smooth surface. Since $(K_{\widehat{T}}+\overline{\delta})l=0$ and $K_{\widehat{T}}l=-1$, we have $K^{2}_{P'}=K^{2}_{\widehat{T}}+1$, $\overline{\delta}^{2}_{P'}=\overline{\delta}^{2}+1$ and $K_{P'}\overline{\delta}_{P'}=K_{\widehat{T}}\overline{\delta}-1$ where $\overline{\delta}_{P'}\doteq \alpha'_{*}(\overline{\delta})$. For $P'$ we also have $\overline{\delta}_{P'}^{2}=\frac{17}{2}+K_{P'}^{2}$ and $K_{P'}\overline{\delta}_{P'}=-2-K_{P'}^{2}$. Since $(K_{\widehat{P}'}+\overline{\delta}_{P'})^{2}=(K_{\widehat{T}}+\overline{\delta})^{2}$ is positive again, we can repeat this process untill we get a birational morphism $\alpha:\widehat{T}\rightarrow P$ contracted some ($-1$)-curves $l_{i}$ with $(K_{\widehat{T}}+\overline{\delta})l_{i}=0$, and there is $\lambda > 1$ such that $\lambda K_{P}+\overline{\delta}_{P}$ is nef. For the same reason, we get $\overline{\delta}_{P}^{2}=\frac{17}{2}+K_{P}^{2}$ and $K_{P}\overline{\delta}_{P}=-2-K_{P}^{2}$ where $\overline{\delta}_{P}\doteq \alpha_{*}(\overline{\delta})$.

If $K^{2}_{P}>0$, then $(\overline{\delta}_{P})^{2}(K_{P})^{2}\leq (\overline{\delta}_{P}K_{P})^{2}$ by index theorem. By 1) we have $K^{2}_{P}\leq \frac{8}{9}< 1$, which is a contradiction. Hence we have $K^{2}_{P} \leq 0$.

Let $\lambda \in \mathbb{Q}\cup \{\infty \}$ be the maximal number such that $\lambda K_{P}+\overline{\delta}_{P}$ is nef. Now we have $\lambda > 1$. By Lemma \ref{extremal ray}, there exists an extremal ray $l$ with $(\lambda K_{P}+\overline{\delta}_{P})l=0$ and $K_{P} l<0$. Since $P$ is neither $\mathds{P}^{2}$ nor a minimal ruled surface, $l$ is a ($-1$)-curve by contraction theorem. So we have $\lambda=\overline{\delta}_{P}l \in \frac{1}{2}\mathbb{Z}$ and $\frac{3}{2}K_{P}+\overline{\delta}_{P}$ is nef.

We have $0\leq (\frac{3}{2}K_{P}+\overline{\delta}_{P})(2K_{P}+\overline{\delta}_{P})=\frac{3}{2}+\frac{1}{2}K^{2}_{P}$. Hence, $K^{2}_{P} \in \{-3,-2,-1,0 \}$.
\end{proof}

\begin{step3}
Let $P_{0}\doteq P$ and $\overline{\delta}_{0}\doteq \overline{\delta}_{P}$. Then there exists a nonnegative integer $s$ and a series of birational morphisms $\beta_{i}:P_{i-1}\rightarrow P_{i} (1\leq i \leq s)$, where $\beta_{i}$ is the contraction of a (-1)-curve $l_{i}$ with $(\frac{3}{2}K_{P_{i-1}}+\overline{\delta}_{i-1})l_{i}=0$ and $\overline{\delta}_{i}\doteq\beta_{i*}(\overline{\delta}_{i-1})$. Moreover,
\begin{enumerate}[1)]
\item $2K_{P_{s}}+\overline{\delta}_{s}$ is nef;
\item If $K_{P}^{2}=0$, then s=0 or 1; If $K_{P}^{2}=-1$, then s=2 or 3; If $K_{P}^{2}=-2$, then s=6; If $K_{P}^{2}=-3$, then s=10.
\end{enumerate}
\end{step3}
\begin{proof}
Let $\lambda \in \mathbb{Q}\cup \{\infty \}$ be the maximal number such that $\lambda K_{P}+\overline{\delta}_{P}$ is nef. By Step 2, we have $\lambda \geq \frac{3}{2}$. If $\lambda = \frac{3}{2}$, then there exists an extremal ray $l_{1}$ with $(\frac{3}{2}K_{P}+\overline{\delta}_{P})l_{1}=0$ and $K_{P} l_{1}<0$ by Lemma \ref{extremal ray}. Since $P$ is neither $\mathds{P}^{2}$ nor a minimal ruled surface, $l$ is a ($-1$)-curve by contraction theorem. So there is a birational morphism $\beta_{1}:P\rightarrow P_{1}$ contracted $l_{1}$, where $P_{1}$ is a smooth surface. Since $(\frac{3}{2}K_{P}+\overline{\delta}_{P})l_{1}=0$ and $K_{P}l_{1}=-1$, we have $K^{2}_{P_{1}}=K^{2}_{P}+1$, $\overline{\delta}^{2}_{P_{1}}=\overline{\delta}^{2}_{P}+\frac{9}{4}$ and  $K_{P_{1}}\overline{\delta}_{P_{1}}=K_{P}\overline{\delta}_{P}-\frac{3}{2}$. In the same way, we can repeat this process untill we get a series of birational morphisms $\beta_{i}:P_{i-1}\rightarrow P_{i} (1\leq i \leq s)$ contracted ($-1$)-curves $l_{i}$ with $(\frac{3}{2}K_{P}+\overline{\delta}_{P})l_{i}=0$, and there is $\lambda > \frac{3}{2}$ such that $\lambda K_{P_{s}}+\overline{\delta}_{P_{s}}$ is nef. With the same calculation method, we have $K^{2}_{P_{s}}=K^{2}_{P}+s$, $\overline{\delta}^{2}_{P_{s}}=\overline{\delta}^{2}_{P}+\frac{9}{4}s$ and $K_{P_{s}}\overline{\delta}_{P_{s}}=K_{P}\overline{\delta}_{P}-\frac{3}{2}s$.

By index theorem, we have $K^{2}_{P_{s}}\leq 7$, which means that $P_{s}$ is neither $\mathds{P}^{2}$ nor a minimal ruled surface. So $2K_{P_{s}}+\overline{\delta}_{P_{s}}$ is nef as before.

We have $0\leq (K_{P_{s}}+\overline{\delta}_{P_{s}})(2K_{P_{s}}+\overline{\delta}_{P_{s}})=\frac{5}{2}-\frac{s}{4}$ and $0\leq (2K_{P_{s}}+\overline{\delta}_{P_{s}})^{2}=K^{2}_{P}+\frac{1}{2}+\frac{s}{4}$. If $s\geq 1$, then $K^{2}_{P_{s}}=K^{2}_{P}+s \geq 1$, which means $(\overline{\delta}_{P_{s}})^{2}(K_{P_{s}})^{2}\leq (\overline{\delta}_{P_{s}}K_{P_{s}})^{2}$ by index theorem. Now we can get $s\leq 2(8-9K^{2}_{P})/ (10+K^{2}_{P})$ for any $ K^{2}_{P}$. So the result of 2) can be calculated.
\end{proof}

\begin{step4}
The following four cases ``$K_{P}^{2}=0, s=1$'', ``$K_{P}^{2}=-1, s=3$'', ``$K_{P}^{2}=-2, s=6$'' and ``$K_{P}^{2}=-3, s=10$'' do not happen.
\end{step4}
\begin{proof}
Let $l \subseteq \widehat{T}$ be a ($-1$)-curve with $(K_{\widehat{T}}+\overline{\delta})l=0$. Since $K_{\widehat{T}}+\overline{\delta}$ is big and nef and $(K_{\widehat{T}}+\overline{\delta})A=0$, Hodge index tells that $Al<\frac{3}{2}$. And by $\overline{\delta}l=-K_{\widehat{T}}l=1$, we have $Al= 2(\widehat{\delta}l-\overline{\delta}l) \in 2\mathbb{Z}$. Thus $Al=0$ and $A$ is still a ($-2$)-curve in $P$.

Let $l\subseteq P$ be a ($-1$)-curve with $(\frac{3}{2}K_{P}+\overline{\delta}_{P})l=0$. Since $\frac{3}{2}K_{P}+\overline{\delta}_{P}$ is also big and nef, we have $Al<\frac{3}{2}$ as before. However, as $\overline{\delta}_{P}l=-\frac{3}{2}K_{P}l=\frac{3}{2}$, we have $Al= 2(\widehat{\delta}l-\overline{\delta}l) \notin 2\mathbb{Z}$. Thus $Al=1$ and $A$ become a ($-1$)-curve contained in $(\frac{3}{2}K_{P}+\overline{\delta}_{P})^{\perp}$ after contracting $l$.

Hence, if $s\geq 1$ then we must have $s\geq 2$. If $s\geq 3$, let $l'$ be another ($-1$)-curve with $(\frac{3}{2}K_{P_{2}}+\overline{\delta}_{P_{2}})l'=0$. For convenience, we do not distinguish $l'$ in  $P_{2}$ from the total inverse image of it in $P$. We also have $(\frac{3}{2}K_{P}+\overline{\delta}_{P})l'=0$ and $Al'=1$. It contradicts that $A$ has been contracted to a point and $l'$ is a curve in $P_{2}$.

In conclusion, we have $s=0$ or $s=2$.
\end{proof}

\bigskip

Now we only need to consider the two cases ``$K_{P}^{2}=0, s=0$'' and ``$K_{P}^{2}=-1, s=2$''. We will give the detailed description of these surfaces and prove the existence in each case.

\begin{pro}
If $K_{P}^{2}=0$ and $s=0$, then $S$ is the minimal resolution of a double cover of a Hirzebruch surface $\mathds{F}_{2}$ branched along a curve $B_{1}$ in $|\Delta_{\infty}|$ and a curve $B_{2}$ in $|9\Delta_{\infty}+18\Gamma|$, where $\Delta_{\infty}$ denotes the section at infinity and $\Gamma$ is a fibre, having eight ordinary 4-tuple points on $B_{2}$ as the only essential singularities.
\end{pro}
\begin{proof}
Let $P\doteq P_{0}$. With the same method, we can get a series of birational morphisms $\beta_{i}:P_{i-1}\rightarrow P_{i} (1\leq i \leq r)$ contracted ($-1$)-curves $l_{i}$ with $(2K_{P}+\overline{\delta}_{P})l_{i}=0$, and there is $\lambda > 2$ such that $\lambda K_{P_{r}}+\overline{\delta}_{P_{r}}$ is nef. In particular, we have $K^{2}_{P_{r}}=K^{2}_{P}+r$, $\overline{\delta}^{2}_{P_{r}}=\overline{\delta}^{2}_{P}+4r$ and $K_{P_{s}}\overline{\delta}_{P_{r}}=K_{P}\overline{\delta}_{P}-2r$.

If $r\geq 1$, then $(\overline{\delta}_{P_{r}})^{2}(K_{P_{r}})^{2}\leq (\overline{\delta}_{P_{r}}K_{P_{r}})^{2}$ by index theorem, which means $r\leq 8$ and $P_{r}$ can't be $\mathds{P}^{2}$. If $P_{r}$ is not a $\mathds{P}^{1}$-bundle, then  $\frac{5}{2}K_{P_{r}}+\overline{\delta}_{P_{r}}$ is nef as before. So we have $0\leq(\frac{5}{2}K_{P_{r}}+\overline{\delta}_{P_{r}}) (2K_{P_{r}}+\overline{\delta}_{P_{r}}) =-\frac{1}{2}$, which is a contradiction.

Hence, $P_{r}$ is a $\mathds{P}^{1}$-bundle and $r=8$. With the method in Step 4, $A$ is a ($-2$)-curve in $P_{8}$ and $P_{8}$ is the Hirzebruch surface $\mathds{F}_{2}$. Let $\lambda \in \mathbb{Q}\cup \{\infty \}$ be the maximal number such that $\lambda K_{P_{8}}+\overline{\delta}_{P_{8}}$ is nef. If $\lambda \geq \frac{5}{2}$, then we also have the contradiction as before. We can assume $2< \lambda < \frac{5}{2}$, then $\lambda=\frac{7}{3}, \frac{9}{4}$ or $\frac{13}{6}$ by rationality theorem (see  \cite{Kollar}). There exists an extremal ray $l$ with $(\lambda K_{P_{8}}+\overline{\delta}_{P_{8}})l=0$ and $K_{P_{8}}\cdot l<0$ by Lemma \ref{extremal ray}. $l$ will be the fibre $\Gamma$ by the contraction theorem and $12(\lambda K_{P_{8}}+\overline{\delta}_{P_{8}})\equiv a\Gamma$ where $a\in \mathbb{Z}$. Besides, we have $K_{P_{8}}\equiv -2\Delta_{0}$, so $12\overline{\delta}_{P_{8}}\equiv 24\lambda \Delta_{0}+a\Gamma$ and
\begin{equation*}
  \begin{cases}
    -2a=(a\Gamma)(-2\Delta_{0})=12(\lambda K_{P_{8}}+\overline{\delta}_{P_{8}})K_{P_{8}}=12(8\lambda-18),  \\
    24\lambda a=(a\Gamma)(24\lambda \Delta_{0}+af)=12(\lambda K_{P_{8}}+\overline{\delta}_{P_{8}}) (12\overline{\delta}_{P_{8}})=144(-18\lambda+\frac{81}{2}).  \\
  \end{cases}
\end{equation*}
So $\lambda=\frac{9}{4}$ and $a=0$, which infers $\widehat{\delta}_{P_{8}}\equiv 5\Delta_{\infty}+9\Gamma$.

Considering the morphism $\eta:P\rightarrow P_{8}$, if we assume that $\widehat{\delta}_{P}\equiv (5\Delta_{\infty}+9\Gamma)-\Sigma_{i=1}^{8}c_{i}E_{i}$, then $c_{i}=\widehat{\delta}_{P}E_{i}=2$. Hence, the branch curve of the double cover $S\dashrightarrow P_{8}$ is the union of  $B_{1}=\Delta_{\infty}$ and $B_{2}$ in $|9\Delta_{\infty}+18\Gamma|$, having eight ordinary 4-tuple points on $B_{2}$ as the only essential singularities.
\end{proof}

\begin{theorem}\label{family 1}
We assume that $f:P\rightarrow \mathds{F}_{2}$ is the blowing-up at 8 points $p_{i}(1\leq i\leq8)$, where all these points are in general position and none of them is in the section at infinity of the Hirzebruch surface $\mathds{F}_{2}$. Let $\Omega \equiv 9\Delta_{\infty}+18\Gamma-4\Sigma_{i=1}^{8}E_{i}$, where $\Delta_{\infty}$ is the section at infinity, $\Gamma$ is a fibre and $E_{i}=f^{-1}(p_{i})$ for $1\leq i\leq8$. Then
\begin{enumerate}[1)]
\item $|\Omega|$ is base point free;
\item The general element in $|\Omega|$ is smooth and irreducible;
\item Let $B$ be a general element in $|\Omega|$ and $A \doteq\Delta_{\infty}$ be the (-2)-curve. If we assume that $S$ is the minimal resolution of a double cover of $P$ branched in $B\cup A$, then $q(S)=0$, $p_{g}(S)=4$ and $ K^{2}_{S}=9$.
\end{enumerate}
\end{theorem}
\begin{proof}
If we define $L\doteq 10\Delta_{\infty}+22\Gamma-5\Sigma_{i=1}^{8}E_{i}$, then $L^{2}=40$ and $L\equiv -5K_{P}+2\Gamma$. By Lemma \ref{base point free} and Lemma \ref{nef 1}, $|-4K_{P}+2\Gamma|$ is base point free. Since $|-4K_{P}+2\Gamma|+\Delta_{\infty} \subset |\Omega|$ and $h^{0}(\Omega)\geq 20 >19= h^{0}(-4K_{P}+2\Gamma)$, there exist an element $D\in |\Omega|$ with $D.\Delta_{\infty} =0$ and $\Delta_{\infty}\nsubseteq D$. Hence, $|\Omega|$ is base point free.

2) is true by Bertini's theorem.

Let $\widehat{S}$ be the finite double cover of $P$ branched in $B\cup A$. By 2) and $B.A=0$, $\widehat{S}$ is a smooth surface. By Lemma 6 in \cite{Horikawa 1}, we have $q(\widehat{S})=0$, $p_{g}(\widehat{S})=4$ and $ K^{2}_{\widehat{S}}=8$. Since there is only one ($-1$)-curve on $\widehat{S}$, which is the preimage of $A$, we have $q(S)=0$, $p_{g}(S)=4$ and $ K^{2}_{S}=9$.
\end{proof}

\begin{pro}
If $K_{P}^{2}=-1$ and $s=2$, then there is a rational number $0\leq r \leq 3$ such that $S$ is the minimal resolution of a double cover of a Hirzebruch surface $\mathds{F}_{r}$ branched along a curve B in $|8\Delta_{\infty}+(10+4r)\Gamma|$, where $\Delta_{\infty}$ denotes the section at infinity and $\Gamma$ is a fibre, having seven ordinary 4-tuple points and one (3,3)-point as the only essential singularities.
\end{pro}
\begin{proof}
We first analyse the birational morphism $\beta:P\rightarrow P_{2}$. With the same argument in Step 4, let $E_{1}$ and $E_{2}$ be the two (-1)-curves, then we have $\overline{\delta}_{P}\equiv \overline{\delta}_{P_{2}}-\frac{3}{2}E_{1}-\frac{3}{2}E_{2}$ and $A\equiv E_{1}-E_{2}$. Since $\beta(A)=pt$, we have $2K_{P_{2}}+\overline{\delta}_{P_{2}} \in Pic(P_{2})$ and the birational morphism $\beta$ corresponds to the canonical resolution of one (3,3)-point in the branch curve.

We have $(2K_{P_{2}}+\overline{\delta}_{P_{2}})^{2}=0$ and $(K_{P_{2}}+\overline{\delta}_{P_{2}})(2K_{P_{2}}+\overline{\delta}_{P_{2}})=2$ by Step 3, $h^{i}(2K_{P_{2}}+\overline{\delta}_{P_{2}})=0$ for any $i>0$ by vanishing theorem, and $h^{0}(2K_{P_{2}}+\overline{\delta}_{P_{2}})=2$ by Riemann-Roch theorem. For any general curve $C\in |2K_{P_{2}}+\overline{\delta}_{P_{2}}|$, we have $g(C)=0$ by the genus formula. Hence, $|2K_{P_{2}}+\overline{\delta}_{P_{2}}|$ is a base point free genus 0 pencil and we have the following commutative graph
$$\xymatrix{
    P_{2} \ar[r]^{\eta}\ar[dr]_{|2K_{P_{2}}+\overline{\delta}_{P_{2}}|} & \mathds{F}_{r} \ar[d] \\
     & \mathds{P}^{1}
    }$$
where $\eta$ is the blowing-up at $K^{2}_{\mathds{F}_{r}}-K^{2}_{P_{2}}=7$ points.

Now we analyse the birational morphism $\eta: P_{2} \rightarrow \mathds{F}_{r}$. Let $B_{\infty}$ be the strict transform of the ($-r$)-section $\Delta_{\infty}$ of $\mathds{F}_{r}$ and $E$ be a ($-1$)-curve contained in a fibre of $|2K_{P_{2}}+\overline{\delta}_{P_{2}}|$, then $B_{\infty}$ is an irreducible rational curve with $B_{\infty}(2K_{P_{2}}+\overline{\delta}_{P_{2}})=1$. Hence, $B_{\infty}E$ is 0 or 1. If $B_{\infty}E=1$, then we have $(2K_{P_{2}}+\overline{\delta}_{P_{2}}-E)^{2}=K_{P_{2}}(2K_{P_{2}}+\overline{\delta}_{P_{2}}-E)=-1$, and $2K_{P_{2}}+\overline{\delta}_{P_{2}}-E$ contains another irreducible ($-1$)-curve $E'$ with $B_{\infty}E'=0$. So we can choose $\eta$ such that all contracted curve E holds $B_{\infty}E=0$.

On this occasion, $B_{\infty}$ is a smooth rational curve with $B_{\infty}^{2}=-r$, then $K_{P_{2}}B_{\infty}=r-2$ by the genus formula. Since $0\leq (K_{P_{2}}+\overline{\delta}_{P_{2}})B_{\infty}=3-r$, we have $r\leq 3$. If we assume that $\overline{\delta}_{P_{2}}\equiv (a\Delta_{\infty}+b\Gamma)-\Sigma_{i=3}^{9}c_{i}E_{i}$, then
\begin{equation*}
  \begin{cases}
    c_{i}=\overline{\delta}_{P_{2}}E_{i}=2(3\leq i\leq9),  \\
    a=\overline{\delta}_{P_{2}}\Gamma=4,  \\
    12=\overline{\delta}_{P_{2}}^{2}=-16r+8b-28, b=5+2r.  \\
  \end{cases}
\end{equation*}
Hence, the branch curve is in $|8\Delta_{\infty}+(10+4r)\Gamma|$, which has one (3,3)-point and seven ordinary 4-tuple points as the only essential singularities.
\end{proof}

If $r \neq 0$, then we can modify the choice of the curves we contract in order to obtain $r = 0$. It follows that the family with $r = 0$ is open and dense in the subscheme of the moduli space of surfaces described in Proposition 4.3 (see Remark 3.6 in \cite{Bauer 2}). So we only consider the case $r=0$ and prove the existence.

\begin{theorem}\label{family 2}
We assume that $f:P\rightarrow \mathds{F}_{0}$ is the blowing-up at 9 points $p_{i}(1\leq i\leq9)$, where $p_{1}$ is in the preimage of $p_{2}$ and there is no other condition for the position of these points. Let $\Omega \equiv 8\Delta_{\infty}+10\Gamma-3(E_{1}+E_{2})-4\Sigma_{i=3}^{9}E_{i}$, where $\Delta_{\infty}$ is the section at infinity, $\Gamma$ is a fibre and $E_{i}=f^{-1}(p_{i})$ for $1\leq i\leq9$. Then
\begin{enumerate}[1)]
\item $|\Omega|$ is base point free;
\item The general element in $|\Omega|$ is smooth and irreducible;
\item Let $B$ be a general element in $|\Omega|$ and $A \doteq E_{2}-E_{1}$ be the (-2)-curve. If we assume that $S$ is the minimal resolution of a double cover of $P$ branched in $B\cup A$, then $q(S)=0$, $p_{g}(S)=4$ and $ K^{2}_{S}=9$.
\end{enumerate}
\end{theorem}
\begin{proof}
If we define $L\doteq 10\Delta_{\infty}+12\Gamma-5\Sigma_{i=1}^{9}E_{i}$, then $L\equiv -5K_{P}+2\Gamma$ and $L^{2}=15$. By the method in appendix, $L$ is nef for the general situation. $|8\Delta_{\infty}+10\Gamma-4\Sigma_{i=1}^{9}E_{i}|$ is base point free by Lemma \ref{base point free}. Since $|8\Delta_{\infty}+10\Gamma-4\Sigma_{i=1}^{9}E_{i}|+E_{1}+E_{2} \subset |\Omega|$, $|\Omega|$ is also base point free by Lemma \ref{base point free 1}.

2) is true by Bertini's theorem. The proof of 3) is similar to that in Theorem 4.2.
\end{proof}

In particular, each surface in Theorem 4.2 has a genus 4 fibration induced by the double cover, while each surface in Theorem 4.4 has a genus 3 fibration.

\bigskip

\section{the case $\tau=3$}

We assume $\tau=3$ in this section. All the methods used in this section are similar to that in Section 4, and the difference is just the numerical value. So the proofs will be described succinctly and the details of the method can be found in Section 4.

\bigskip

\begin{step1}
$K_{\widehat{T}}+\overline{\delta}$ is nef and $2K_{\widehat{T}}+\overline{\delta}$ is $\mathbb{Q}$-effective.
\end{step1}
\begin{proof}
See Section 4.
\end{proof}

\begin{step2}
There exists a smooth surface $P$ and a birational morphism $\alpha:\widehat{T}\rightarrow P$ contracting some (-1)-curves $l_{i}$ with $(K_{\widehat{T}}+\overline{\delta})l_{i}=0$. Let $\overline{\delta}_{P}\doteq \alpha_{*}(\overline{\delta})$. Then
\begin{enumerate}[1)]
\item $\overline{\delta}_{P}^{2}=\frac{21}{2}+K_{P}^{2}$ and $K_{P}\overline{\delta}_{P}=-3-K_{P}^{2}$;
\item $\frac{3}{2}K_{P}+\overline{\delta}_{P}$ is nef;
\item $K_{P}^{2}\in \{0,1,2\}$.
\end{enumerate}
\end{step2}
\begin{proof}
For the same argument in Section 4, we have $\overline{\delta}^{2}=\frac{21}{2}+K_{\widehat{T}}^{2}$ and
$K_{\widehat{T}}\overline{\delta}=-3-K_{\widehat{T}}^{2}$. Besides, we also get a birational morphism $\alpha:\widehat{T}\rightarrow P$ contracting some ($-1$)-curves $l_{i}$ with $(K_{\widehat{T}}+\overline{\delta})l_{i}=0$ and there is $\lambda >1$ such that $\lambda K_{P}+\overline{\delta}_{P}$ is nef. Moreover, $\overline{\delta}_{P}^{2}=\frac{21}{2}+K_{P}^{2}$ and $K_{P}\overline{\delta}_{P}=-3-K_{P}^{2}$.

If $K^{2}_{P}>0$, then $(\overline{\delta}_{P})^{2}(K_{P})^{2}\leq (\overline{\delta}_{P}K_{P})^{2}$ by index theorem. By 1), we have $K^{2}_{P}\leq 2$. Since $P$ is neither $\mathds{P}^{2}$ nor a minimal ruled surface, we also have that $\frac{3}{2}K_{P}+\overline{\delta}_{P}$ is nef.

By Step 1, we have $0\leq (\frac{3}{2}K_{P}+\overline{\delta}_{P})(2K_{P}+\overline{\delta}_{P})=\frac{1}{2}K^{2}_{P}$. Hence, $K^{2}_{P} \in \{0,1,2 \}$.
\end{proof}

\begin{step3}
Let $P_{0}\doteq P$ and $\overline{\delta}_{0}\doteq \overline{\delta}_{P}$. Then there exists a nonnegative integer $s$ and a series of birational morphisms $\beta_{i}:P_{i-1}\rightarrow P_{i} (1\leq i \leq s)$, where $\beta_{i}$ is the contraction of a (-1)-curve $l_{i}$ with $(\frac{3}{2}K_{P_{i-1}}+\overline{\delta}_{i-1})l_{i}=0$ and $\overline{\delta}_{i}\doteq\beta_{i*}(\overline{\delta}_{i-1})$. Moreover,
\begin{enumerate}[1)]
\item $2K_{P_{s}}+\overline{\delta}_{s}$ is nef;
\item If $K_{P}^{2}=2$, then s=0; If $K_{P}^{2}=1$, then s=2; If $K_{P}^{2}=0$, then s=6.
\end{enumerate}
\end{step3}
\begin{proof}
As before, we get a series of birational morphisms $\beta_{i}:P_{i-1}\rightarrow P_{i} (1\leq i \leq s)$ contracting some ($-1$)-curves $l_{i}$ with $(\frac{3}{2}K_{P}+\overline{\delta}_{P})l_{i}=0$, and there is $\lambda >\frac{3}{2}$ such that $\lambda K_{P}+\overline{\delta}_{P}$ is nef. Moreover, $K^{2}_{P_{s}}=K^{2}_{P}+s$, $\overline{\delta}^{2}_{P_{s}}=\overline{\delta}^{2}_{P}+\frac{9}{4}s$ and $K_{P_{s}}\overline{\delta}_{P_{s}}=K_{P}\overline{\delta}_{P}-\frac{3}{2}s$. We can also prove that $P_{s}$ is neither $\mathds{P}^{2}$ nor a minimal ruled surface and  $2K_{P_{s}}+\overline{\delta}_{P_{s}}$ is nef.

By Step 1, we have $0\leq (2K_{P_{s}}+\overline{\delta}_{P_{s}})^{2}=K^{2}_{P}-\frac{3}{2}+\frac{s}{4}$. If $s>0$, then $K_{P_{s}}^{2}>0$ and $(\overline{\delta}_{P_{s}})^{2}(K_{P_{s}})^{2}\leq (\overline{\delta}_{P_{s}}K_{P_{s}})^{2}$ by index theorem.  So the result of 2) can be calculated.
\end{proof}

\bigskip

Now we rebuild the structure of the surface in Step 3 separately.

\bigskip

\begin{pro}
If $K_{P}^{2}=2$, then $S$ is the minimal resolution of a double cover of a weak Del Pezzo surface of degree 2 branched along a curve $B$ in $|-5K|$ and three (-2)-curves $A_{i}(1\leq i \leq 3)$ with $BA_{i}=A_{i}A_{j}=0(1\leq i<j \leq 3)$, having no singular point as the only essential singularities.
\end{pro}
\begin{proof}
If $K_{P}^{2}=2$, then $s=0$ and $25=(\overline{\delta}_{P})^{2}(K_{P})^{2}\leq (\overline{\delta}_{P}K_{P})^{2}=25$, which infers that $5K_{P}+2\overline{\delta}_{P}$ is numberial trival by index theorem. Since $4K_{P}+2\overline{\delta}_{P}$ is big and nef, we have  $h^{i}(5K_{P}+2\overline{\delta}_{P})=0$ $(i>0)$ by vanishing theorem and $h^{0}(5K_{P}+2\overline{\delta}_{P})=1$ by Riemann-Roch theorem. Hence, $5K_{P}+2\overline{\delta}_{P}$ is trival.

Besides, $-K_{P}\equiv 4K_{P}+2\overline{\delta}_{P}$ is big and nef, so $P$ is a weak Del Pezzo surface of degree $K^{2}_{P}=2$. As in Step 4 of Section 4, $A_{i}(1\leq i \leq 3)$ is also ($-2$)-curves in $P$, and the other branch curve $B$ is in $|-5K_{P}|$ with $BA_{i}=A_{i}A_{j}=0(1\leq i<j \leq 3)$.
\end{proof}

For a general surface in this family, we can assume that $P$ is the blowing-up of $\mathds{P}^{2}$ at 7 points $p_{i}(1\leq i\leq 7)$, and $p_{2k-1}$ is in the preimage of $p_{2k}$ where $k=1,2,3$.

\begin{theorem}\label{family 3}
We assume that $f:P\rightarrow \mathds{P}^{2}$ is the blowing-up at 7 points $p_{i}(1\leq i\leq7)$, where $p_{2k-1}$ is in the preimage of $p_{2k}(k=1,2,3)$ and there is no other condition for the position of these points. Let $\Omega \equiv -5K_{P}$. Then
\begin{enumerate}[1)]
\item $|\Omega|$ is base point free;
\item The general element in $|\Omega|$ is smooth and irreducible;
\item Let $B$ be a general element in $|\Omega|$ and $A_{k} \doteq E_{2k}-E_{2k-1}(k=1,2,3)$ be the (-2)-curves. If we assume that $S$ is the minimal resolution of a double cover of $P$ branched in $B\cup_{k} A_{k}$, then $q(S)=0$, $p_{g}(S)=4$ and $ K^{2}_{S}=9$.
\end{enumerate}
\end{theorem}
\begin{proof}
If we define $L\doteq -6K_{P}$, then $L^{2}=72$ and $\Omega \equiv K_{P}+L$. By Lemma \ref{base point free} and Lemma \ref{nef 1}, $|\Omega|$ is base point free.

2) is true by Bertini's theorem. The proof of 3) is similar to that in Theorem 4.2.
\end{proof}

\begin{pro}
If $K_{P}^{2}=1$, then there is a rational number $0\leq r \leq 2$ such that S is the minimal resolution of a double cover of a Hirzebruch surface $\mathds{F}_{r}$ branched along a fibre $B_{1}$ in $|\Gamma|$ and a curve $B_{2}$ in $|8\Delta_{\infty}+(9+4r)\Gamma|$, where $\Delta_{\infty}$ denotes the section at infinity and $\Gamma$ is a fibre. $B_{2}$ has one (3,3)-point and five singular points $x_{1},\cdots,x_{5}$ of multiplicity 4 as the only essential singularities, where $x_{4}$ is in $B_{1}$ and $x_{5}$ is infinitely near to $x_{4}$ belonging to the strict transform of $B_{1}$.
\end{pro}
\begin{proof}
We first analyse the birational morphism $\beta:P\rightarrow P_{2}$. Let $E_{1}$ and $E_{2}$ be the two ($-1$)-curves. By the argument in Step 4 of Section 4, we can assume $A_{1}\equiv E_{1}-E_{2}$ and $\overline{\delta}_{P}\equiv\overline{\delta}_{P_{2}}-\frac{3}{2}E_{1}-\frac{3}{2}E_{2}$. Now we have $2(2K_{P_{2}}+\overline{\delta}_{P_{2}}) \in Pic(P_{2})$ and the birational morphism $\beta$ corresponds to the canonical resolution of one (3,3)-point in the branch curve.

If $K_{P}^{2}=1$, then $s=2$, which infers $(2K_{P_{2}}+\overline{\delta}_{P_{2}})^{2}=0$ and $(K_{P_{2}}+\overline{\delta}_{P_{2}})(2K_{P_{2}}+\overline{\delta}_{P_{2}})=1$. Since $3K_{P_{2}}+2\overline{\delta}_{P_{2}}$ is big and nef, we have $h^{i}(4K_{P_{2}}+2\overline{\delta}_{P_{2}})=0(i>0)$ by vanishing theorem and $h^{0}(4K_{P_{2}}+2\overline{\delta}_{P_{2}})=2$ by Riemann-Roch theorem. We have $g(C)=0$ by the genus formula for any general curve $C\in |2(2K_{P_{2}}+\overline{\delta}_{P_{2}})|$. Hence,  $|2(2K_{P_{2}}+\overline{\delta}_{P_{2}})|$ is a base point free genus 0 pencil and we have the following commutative graph
$$\xymatrix{
    P_{2} \ar[r]^{\eta}\ar[dr]_{|2(2K_{P_{2}}+\overline{\delta}_{P_{2}})|} & \mathds{F}_{r} \ar[d] \\
     & \mathds{P}^{1}
    }$$
where $\eta$ is the blowing-up at $K^{2}_{\mathds{F}_{r}}-K^{2}_{P_{2}}=5$ points.

Now we analyse the birational morphism $\eta: P_{2} \rightarrow \mathds{F}_{r}$, we can choose $\eta$ such that all contracted curve E holds $B_{\infty}E=0$ with the same argument in Proposition 4.3. So $B_{\infty}$ is a smooth rational curve with $B_{\infty}^{2}=-r$, then we can get $K_{P_{2}}B_{\infty}=r-2$ by the genus formula. Since $0\leq (K_{P_{2}}+\overline{\delta}_{P_{2}})B_{\infty}=\frac{5}{2}-r$, then $r\leq 2$. Let $l$ be one of these 5 ($-1$)-curves. Since $l,A_{2},A_{3}$ all contained in fibres, then $lA_{2}\leq1$ and $lA_{3}\leq1$ by Zariski's lemma. Moreover, it cann't be $lA_{2}=lA_{3}=0$ for all $l$, since $\mathds{F}_{r}$ does not contain two disjoint ($-2$)-curves. We assume $E_{3}A_{2}=1$ where $E_{3}$ is the ($-1$)-curve. Since $E_{3}(A_{2}+A_{3})=2(E_{3}\widehat{\delta}-E_{3}\overline{\delta})$ is even, then $E_{3}A_{2}=E_{3}A_{3}=1$. After the contraction of $E_{3}$, $A_{2}$ and $A_{3}$ become  ($-1$)-curves in a fibre with $A_{2}A_{3}=1$. Hence,  one will be contracted and the other will map isomorphically onto a fibre of $\mathds{F}_{r}$, which we denote $B_{1}$.

Let $\overline{\delta}_{P_{2}}\equiv (a\Delta_{\infty}+b\Gamma)-\Sigma_{i=3}^{7}c_{i}E_{i}$. Since $\Gamma \equiv 2(2K_{P_{2}}+\overline{\delta}_{P_{2}})$, then
\begin{equation*}
  \begin{cases}
    c_{i}=\overline{\delta}_{P_{2}}E_{i}=2 (3\leq i\leq 7),  \\
    a=\overline{\delta}_{P_{2}} \Gamma=4,  \\
    16=\overline{\delta}_{P_{2}}^{2}=-16r+8b-20, b=\frac{9}{2}+2r.  \\
  \end{cases}
\end{equation*}
Hence, the branch curve $B_{2}$ is in $|8\Delta_{\infty}+(10+4r)\Gamma|$, which has one (3,3)-point and five singular points $x_{1},\cdots,x_{5}$ of multiplicity 4 as the only essential singularities. Besides,  $x_{4}$ is in $B_{1}$ and $x_{5}$ is infinitely near to $x_{4}$ belonging to the strict transform of $B_{1}$.
\end{proof}

As before, the family with $r = 0$ is open and dense in the subscheme of the moduli space of surfaces described in Proposition 5.3.

\begin{theorem}\label{family 4}
We assume that $f:P\rightarrow \mathds{F}_{0}$ is the blowing-up at 7 points $p_{i}(1\leq i \leq 7)$, where $p_{1}$ is in the preimage of $p_{2}$, and $p_{3}$ is the intersection of the preimage of $p_{4}$ and the strict transform of the fibre $B_{1}$ passing through $p_{4}$. Let $\Omega \equiv 8\Delta_{\infty}+9\Gamma-(3E_{1}+3E_{2})-\Sigma_{i=3}^{7}4E_{i}$, where $\Delta_{\infty}$ is the section at infinity, $\Gamma$ be a fibre and $E_{i}=f^{-1}(p_{i})(1\leq i\leq7)$. Then
\begin{enumerate}[1)]
\item $|\Omega|$ is base point free;
\item The general element in $|\Omega|$ is smooth and irreducible;
\item Let $B$ be a general element in $|\Omega|$ and $A_{k} \doteq E_{2k}-E_{2k-1}(k=1,2)$, $A_{3} \doteq B_{1}-E_{4}-E_{3}$ be the (-2)-curves. If we assume that $S$ is the minimal resolution of a double cover of $P$ branched in $B\cup_{k} A_{k}$, then $q(S)=0$, $p_{g}(S)=4$ and $ K^{2}_{S}=9$.
\end{enumerate}
\end{theorem}
\begin{proof}
If we define $L\doteq -5K_{P}$, then $L^{2}=5$ and $8\Delta_{\infty}+8\Gamma-4\Sigma_{i=1}^{7}E_{i} \equiv K_{P}+L$. By Lemma \ref{base point free} and Lemma \ref{nef 1}, $|8\Delta_{\infty}+8\Gamma-4\Sigma_{i=1}^{7}E_{i}|$ is base point free. Since we have $|8\Delta_{\infty}+8\Gamma-4\Sigma_{i=1}^{7}E_{i}|+|\Gamma|+E_{1}+E_{2} \subset |\Omega|$, $|\Omega|$ is also base point free by Lemma \ref{base point free 1}.

2) is true by Bertini's theorem. The proof of 3) is similar to that in Theorem 4.7.
\end{proof}

\begin{pro}
If $K_{P}^{2}=0$, then S is the minimal resolution of a double cover of a weak Del Pezzo surface of degree 6 branched in a curve B in $|-4K|$, having three (3,3)-points as the only essential singularities.
\end{pro}
\begin{proof}
If $K_{P}^{2}=0$, then $s=6$. Since $(2K_{P_{6}}+\overline{\delta}_{P_{6}})^{2}= (2K_{P_{6}}+\overline{\delta}_{P_{6}})(K_{P_{6}}+\overline{\delta}_{P_{6}})=0$ and $(K_{P_{6}}+\overline{\delta}_{P_{6}})^{2}=6$, $2K_{P_{6}}+\overline{\delta}_{P_{6}}$ is numberial trival by index theorem. By Step 1, we have $2K_{P_{6}}+\overline{\delta}_{P_{6}}$ is trival. Now, $-K_{P_{6}}\equiv K_{P_{6}}+\overline{\delta}_{P_{6}}$ is big and nef, so $P_{6}$ is a weak Del Pezzo Surface of degree $K_{P}^{2}+6=6$.

We analyse the birational morphism $\beta:P\rightarrow P_{6}$. Let $E_{i}(1\leq i \leq6)$ be the six ($-1$)-curves. By the argument in Step 4 of Section 4, we can assume $A_{k} \doteq E_{2k}-E_{2k-1}(k=1,2,3)$ and $\overline{\delta}_{P}\equiv \overline{\delta}_{P_{2}}-\Sigma_{i=1}^{6}\frac{3}{2}E_{i}$. Since $\beta(A_{i})=pt$, we have $2K_{P_{6}}+\overline{\delta}_{P_{6}} \in Pic(P_{6})$ and the birational morphism $\beta$ corresponds to the canonical resolution of three (3,3)-points in the branch curve.
\end{proof}

For a general surface in this family, we can assume that P is the blowing-up of $\mathds{P}^{2}$ at 9 points $p_{i}(1\leq i\leq 9)$, and $p_{2k-1}$ is in the preimage of $p_{2k}$ where $k=1,2,3$.

\begin{theorem}\label{family 5}
We assume that $f:P \rightarrow \mathds{P}^{2}$ is the blowing-up at 9 points $p_{i}(1\leq i\leq9)$, where $p_{2k-1}$ is in the preimage of $p_{2k}(k=1,2,3)$ and there is no other condition for the position of these points. Let $P_{6}$ be the surface blowing-up of $\mathds{P}^{2}$ at $p_{i}(i=7,8,9)$ and $\Omega \equiv -4K_{P_{6}}-3\Sigma_{i=1}^{6} E_{i} \equiv 12h -3\Sigma_{i=1}^{6} E_{i} -4\Sigma_{i=7}^{9} E_{i}$, where $h$ is the line in $\mathds{P}^{2}$. Then
\begin{enumerate}[1)]
\item $|\Omega|$ is base point free;
\item The general element in $|\Omega|$ is smooth and irreducible;
\item Let $B$ be a general element in $|\Omega|$ and $A_{k} \doteq E_{2k}-E_{2k-1}(k=1,2,3)$ be the (-2)-curves. If we assume that $S$ is the minimal resolution of a double cover of $P$ branched in $B\cup_{k} A_{k}$, then $q(S)=0$, $p_{g}(S)=4$ and $ K^{2}_{S}=9$.
\end{enumerate}
\end{theorem}
\begin{proof}
Let $L\doteq 13h-3\Sigma_{i=1}^{4}E_{i}-5\Sigma_{i=5}^{9}E_{i}\equiv -3K_{P}+2(h-\Sigma_{i=5}^{9}E_{i})$. we have $L^{2}=8$ and $10h-2\Sigma_{i=1}^{4}E_{i}-4\Sigma_{i=5}^{9}E_{i} \equiv K_{P}+L$. By the method in appendix, $L$ is nef for the general situation. By Lemma \ref{base point free}, $|10h-2\Sigma_{i=1}^{4}E_{i}-4\Sigma_{i=5}^{6}E_{i}-4\Sigma_{i=7}^{9}E_{i}|$ is base point free. Similarly, $|10h-4\Sigma_{i=1}^{2}E_{i}-2\Sigma_{i=3}^{6}E_{i}-4\Sigma_{i=7}^{9}E_{i}|$ is also base point free.

Now we find two subsystems of $|\Omega|$ as follows
\begin{equation*}
  \begin{cases}
    |10h-2\Sigma_{i=1}^{4}E_{i}-4\Sigma_{i=5}^{6}E_{i}-4\Sigma_{i=7}^{9}E_{i}|+
    |2h-\Sigma_{i=1}^{4}E_{i}|+E_{5}+E_{6} \subset |\Omega|,  \\
    |10h-4\Sigma_{i=1}^{2}E_{i}-2\Sigma_{i=3}^{6}E_{i}-4\Sigma_{i=7}^{9}E_{i}|+
    |2h-\Sigma_{i=3}^{6}E_{i}|+E_{1}+E_{2} \subset |\Omega|.  \\
  \end{cases}
\end{equation*}
Since $|2h-\Sigma_{i=1}^{4}E_{i}|$ has no base point on $E_{i}(i=1,2)$ and $|2h-\Sigma_{i=3}^{6}E_{i}|$ has no base point on $E_{j}(j=5,6)$ and $E_{i}.E_{j}=0(i=1,2;j=5,6)$, $|\Omega|$ is base point free.

2) is true by Bertini's theorem. The proof of 3) is similar to that in Theorem 4.7.
\end{proof}

\bigskip

In particular, each surface in Theorem 5.2 has a genus 4 fibration induced by the double cover, while each surface in Theorem 5.4 and Theorem 5.6 has a genus 3 fibration.

\bigskip

\section{the case $\tau=5$}

In this section, we will prove that $\tau=5$ if and only if $S$ has a genus 2 fibration. Based on the methods of Horikawa in \cite{Horikawa 2} and of Catanese and Pignatelli in \cite{Catanese}, we find all the satisfying surfaces.

\bigskip

\begin{pro}
If $\tau=5$, then we have an unique genus 2 fibration $f:S\rightarrow \mathds{P}^{1}$, and the involution on each fibre induces an involution on $S$ such that both the canonical map and the bicanonical map of $S$ factor through it.
\end{pro}
\begin{proof}
If $\tau=5$, then $h^{0}(2K_{\widehat{T}}+\widehat{\delta})=0$ by Lemma \ref{main lemma} 6). Since $\pi_{*}\mathcal{O}_{\widehat{S}}\simeq\mathcal{O}_{\widehat{T}}\oplus \mathcal{O}_{\widehat{T}}(-\widehat{\delta})$ and $K_{\widehat{S}}\equiv \pi^{*}(K_{\widehat{T}}+\widehat{\delta})$, we have $H^{0}(\widehat{S},2K_{\widehat{S}})\simeq H^{0}(\widehat{T},2K_{\widehat{T}}+2\widehat{\delta})$ and the two canonical map of $\widehat{S}$ factors through the involution.

On this condition, the bicanonical map of $S$ is nonbirational, which infers that $S$ have a genus 2 fibration $f:S \rightarrow \mathds{P}^{1}$ by Theorem 1.8, Theorem 2.1 in \cite{Ciliberto} and Theorem 5 in \cite{Horikawa 2}. If $|F_{1}|$ and $|F_{2}|$ are two different pencils of genus 2 without base point on $S$, then $F_{1}F_{2}\geq2$ and $(F_{1}+F_{2})^{2}\geq4>0$. We have $K_{S}(F_{1}+F_{2})=4$ by genus formula and $K_{S}^{2}\leq4$ by index theorem. This contradicts $K^{2}_{S}=9$. Hence, the genus 2 fibration on $S$ is unique.

Let $F$ be a general fibre of the genus 2 fibration. Since $h^{0}(S,K_{S})=4$ and $h^{0}(F,K_{F})=2$, we have $h^{0}(S,K_{S}-F)\geq2$. There exists $D\in |K_{S}-F|$ and D is 1-connected with $D^{2}>0$. So $h^{1}(S,-D)=0$ by Theorem A in \cite{Bombieri}. Since $h^{1}(S,2K_{S}-F)=h^{1}(S,-D)=0$ by Serre duality, the restriction map $H^{0}(S,2K_{S})\rightarrow H^{0}(F,2K_{F})$ is surjective. It's obvious that $|2K_{F}|$ factors through the involution on the general fibre $F$, so we can know that $|2K_{S}|$ factors the involution induced by the fibre.
\end{proof}

The involution on $S$ induced by the genus 2 fibration is exaclty the involution $i$. Moreover, if $S$ has a genus 2 fibration, then $i$ factors through the bicanonical map, which means $h^{0}(2K_{\widehat{T}}+\widehat{\delta})=0$ and $\tau=5$. Hence, none of the surfaces with $\tau=1$ or $\tau=3$ has a genus 2 pencil.

\begin{pro}
Let $S$ be a minimal surface with $p_{g}(S)=4,q(S)=0$ having a genus 2 fibration $f:S\rightarrow \mathds{P}^{1}$. If $|K_{S}|$ is not composed with a pencil, then the canonical map of S factors through the relative canonical map, and $\mathds{P}(f_{*}w_{S|\mathds{P}^{1}})\simeq \mathds{F}_{r}$ with $r=0$ or 2.
\end{pro}
\begin{proof}
Let $\phi$ be the projection $\mathds{P}(f_{*}w_{S})\rightarrow \mathds{P}^{1}$ and $\mathcal{O}(1)$ be the tautological bundle of $\mathds{P}(f_{*}w_{S})$. Since $f_{*}w_{S}\simeq \phi_{*}\mathcal{O}(1)$, we have $H^{0}(S,w_{S})\simeq  H^{0}(\mathds{P}(f_{*}w_{S}),\mathcal{O}(1))$. So $\varphi_{K_{S}}=\varphi\circ\theta$, where $\theta: \mathds{P}(f_{*}w_{S}) \dashrightarrow \mathds{P}^{3}$ is defined by $\mathcal{O}(1)$ and $\varphi :S\dashrightarrow \mathds{P}(f_{*}w_{S})$.

Since $h^{1}(S,K_{S})=q(S)=0$, the cokernel of the restriction map $H^{0}(S,K_{S})\rightarrow H^{0}(F,K_{F})$ is $H^{1}(S,K_{S}-F)$ for any fibre $F$. If we assume $h^{1}(S,K_{S}-F)>0$, then the restriction map $H^{0}(S,K_{S}-iF)\rightarrow H^{0}(F,K_{F})$ is not surjective and  $h^{0}(S,K_{S}-iF)\leq h^{0}(S,K_{S}-(i+1)F)+1$ for any $i\geq 0$. Hence, there is a positive integer $n$ with $h^{0}(S,K_{S})=h^{0}(S,nF)=n+1$. This contradicts that $|K_{S}|$ is not composed with a pencil. Thus, $h^{1}(S,K_{S}-F)=0$ and the restriction map $H^{0}(S,K_{S})\rightarrow H^{0}(F,K_{F})$ is surjective for any fibre $F$.

Let $C\doteq\varphi(F)$. Since $H^{0}(\mathbb{P}(f_{*}w_{S}), \mathcal{O}(1)) \rightarrow H^{0}(C,\mathcal{O}_{C}(1))$ is surjective, the map $\theta|_{C}$ is an embedding morphism for any fibre $C$, which infers that $\theta$ is base point free. Since $\varphi_{K_{S}}$ is not composed with a pencil and $\varphi_{K_{S}}=\varphi\circ\theta$, $\theta$ is a birational morphism.

Let $H=\theta(C)$. We have $\mathrm{deg}H=\mathrm{deg}\varphi_{K_{S}}(S)$ and the map $\theta|_{C}:C \rightarrow H$ is an embedding morphism. Since $H$ is a plane curve, we have $0=g(H)=\frac{1}{2}(\mathrm{deg}H-1) (\mathrm{deg}H-2)$. Hence, $\mathrm{deg}\varphi_{K_{S}}(S)=2$ and $\mathds{P}(f_{*}w_{S})$ is the minimal resolution of the singularities of $\varphi_{K_{S}}(S)$, a fortiori, $\mathds{P}(f_{*}w_{S})\simeq \mathds{F}_{r}$ with $r=0$ or 2.
\end{proof}

\begin{remark}
In the case of $K^{2}_{S}=9$, we have another two solutions to prove $deg\varphi_{K_{S}}(S)=2$ or $\mathds{P}(f_{*}w_{S})\simeq \mathds{F}_{r}$ with $r=0$ or 2.

\begin{enumerate}[1)]
\item Since $\tau=5$, the canonical system has at least 2 base points or a base part by Lemma \ref{canonical involution}. As $\mathrm{deg}\varphi_{K_{S}}(S)\leq \frac{1}{2}(K_{S}^{2}-2) < 4$, then the canonical map is of degree two onto a cubic or a quadric. If $\varphi_{K_{S}}(S)$ is a cubic, then the singular locus of $\varphi_{K_{S}}(S)$ is of dimension one. We can find a contradiction according to the proof of Lemma 3.14 in \cite{Bauer 1}.
\item Let $f_{*}w_{S}=\mathcal{O}(a)\oplus \mathcal{O}(b)$ and $a\leq b$. Since $\mathrm{deg} f_{*}w_{S}=a+b=2$ and $h^{0}(\mathds{P}^{1},f_{*}w_{S})=4$, there are three cases ``$a=1,b=1$'', ``$a=0,b=2$'' and ``$a=-1,b=3$''. By studying the singularities of the branch curve of the double cover $S \dashrightarrow \mathds{P}(f_{*}w_{S|\mathds{P}^{1}})$, we can find a contradiction if ``$a=-1,b=3$''.
\end{enumerate}
\end{remark}

In order to finish the classification, we use the notations and the methods of Horikawa in \cite{Horikawa 1} and \cite{Horikawa 2}.

\begin{theorem}\label{family 6}
If $\tau=5$, then S is the minimal resolution of a double cover of a Hirzebruch surface $\mathds{F}_{r}(r=0,2)$ branched along a curve in $|6\Delta+(6+2k)\Gamma|(k=0,1,2)$, where $\Delta$ denotes the diagonal or the 0-section according as r=0 or 2 and $\Gamma$ is a fibre, with the possible singularities shown in the following table.

\begin{tabular}{m{1cm} m{0.2cm} m{7cm} m{1.5cm} m{1.7cm}}
\hline
Family & r & Conditions of Singularities & Branch curves & Dimension \\
\hline
$\mathcal{M}_{1,0}$ & 0 & $v(I_{1})+v(III_{1})+v(V)=5$ & $6\Delta+10\Gamma$  & 32 \\
$\mathcal{M}_{1,2}$ & 2 & $v(I_{1})+v(III_{1})+v(V)=5$ & $6\Delta+10\Gamma$  & 31 \\
\hline
$\mathcal{M}_{2,0}$ & 0 & $v(I_{1})+v(III_{1})+v(V)=3,v(II_{1})+v(IV_{1})=1$ & $6\Delta+8\Gamma$  & 31 \\
$\mathcal{M}_{2,2}$ & 2 & $v(I_{1})+v(III_{1})+v(V)=3,v(II_{1})+v(IV_{1})=1$ & $6\Delta+8\Gamma$  & 30 \\
\hline
$\mathcal{M}_{3,0}$ & 0 & $v(I_{1})+v(III_{1})+v(V)=2,v(I_{2})+v(III_{2})=1$ & $6\Delta+8\Gamma$  & 30 \\
$\mathcal{M}_{3,2}$ & 2 & $v(I_{1})+v(III_{1})+v(V)=2,v(I_{2})+v(III_{2})=1$ & $6\Delta+8\Gamma$  & 29 \\
\hline
$\mathcal{M}_{4,0}$ & 0 & $v(I_{1})+v(III_{1})+v(V)=1,v(II_{1})+v(IV_{1})=2$ & $6\Delta+6\Gamma$  & 30 \\
$\mathcal{M}_{4,2}$ & 2 & $v(I_{1})+v(III_{1})+v(V)=1,v(II_{1})+v(IV_{1})=2$ & $6\Delta+6\Gamma$  & 29 \\
\hline
$\mathcal{M}_{5,0}$ & 0 & $v(I_{1})+v(III_{1})+v(V)=1,v(II_{2})+v(IV_{2})=1$ & $6\Delta+6\Gamma$  & 29 \\
$\mathcal{M}_{5,2}$ & 2 & $v(I_{1})+v(III_{1})+v(V)=1,v(II_{2})+v(IV_{2})=1$ & $6\Delta+6\Gamma$  & 28 \\
\hline
$\mathcal{M}_{6,0}$ & 0 & $v(I_{3})+v(III_{3})=1$ & $6\Delta+6\Gamma$  & 28 \\
$\mathcal{M}_{6,2}$ & 2 & $v(I_{3})+v(III_{3})=1$ & $6\Delta+6\Gamma$  & 27 \\
\hline
\end{tabular}
\end{theorem}

\begin{proof}
By Lemma 3.1 2), $|K_{S}|$ is not composed with a pencil. By theorem 1 in \cite{Horikawa 2} and Proposition 6.1, the branch curve $B$ of the double cover $S \dashrightarrow \mathds{P}(f_{*}w_{S|\mathds{P}^{1}})$ only has the following singularities $(0), (I_{k}), (II_{k}) , (III_{k}), (IV_{k})$ and $(V)$ with $k\geq1$. By Proposition 6.2, $\mathds{P}(f_{*}w_{S|\mathds{P}^{1}})\simeq \mathds{F}_{0}$ or $\mathds{F}_{2}$. Since $5=\sum_{k}\{(2k-1)(v(I_{k})+v(III_{k}))+2k(v(II_{k})+v(IV_{k}))\}+v(V)$ by theorem 3 in \cite{Horikawa 2}, the possible singularities are shown in the table. By Lemma 6 in \cite{Horikawa 1} and the property of surface with a genus 2 fibration, we can find the equivalence class of branch curve for each family.

As a representative, we prove the existence of surface in $\mathcal{M}_{1,0}$. Let $\Delta_{\infty}$ be the section with $\Delta_{\infty}^{2}=0$ on $\mathds{F}_{0}$. We assume that $P$ is the blowing-up of $\mathds{F}_{0}$ at 10 points $p_{i}(1\leq i\leq10)$, where $p_{2k-1}$ and  $p_{2k}$ are in fibre $B_{k}$ for $1\leq k \leq 5$. If we define $f:P\rightarrow \mathds{F}_{0}$ and $\Omega \equiv 6\Delta+5\Gamma-3\Sigma_{i=1}^{10} E_{i}$, then we can prove that $|\Omega|$ is base point free.

In fact, let $L\doteq 6\Delta+2\Gamma-3\Sigma_{i=1}^{10} E_{i}$, we have $L^{2}=6$ and $4\Delta+2\Gamma-2\Sigma_{i=1}^{10} E_{i} \equiv K_{P}+L$. By the method in appendix, $L$ is nef for the general situation. By Lemma \ref{base point free}, $|4\Delta+2\Gamma-2\Sigma_{i=1}^{10} E_{i}|$ is base point free. Similarly, $|4\Delta+2\Delta_{\infty}-2\Sigma_{i=1}^{10} E_{i}|$ is also base point free. Let $\widehat{B}_{k}$ be the strict transform of $B_{k}$ and $\widehat{\Delta}_{i,j,k}$ be the strict transform of the section passing through the points $p_{i}$, $p_{j}$ and $p_{k}$. Now we find three subsystems of $|\Omega|$ as follows
\begin{equation*}
  \begin{cases}
    |4\Delta+2\Delta_{\infty}-2\Sigma_{i=1}^{10} E_{i}|+
    |2\Gamma|+\Sigma_{k=1}^{5}\widehat{B}_{k} \subset |\Omega|,  \\
    |4\Delta+2\Gamma-2\Sigma_{i=1}^{10} E_{i}|+
    |\Gamma|+\widehat{\Delta}_{1,3,5}+\widehat{\Delta}_{2,4,6}+\widehat{B}_{4}+\widehat{B}_{5} \subset |\Omega|,  \\
    |4\Delta+2\Gamma-2\Sigma_{i=1}^{10} E_{i}|+
    |\Gamma|+\widehat{\Delta}_{5,7,9}+\widehat{\Delta}_{6,8,10}+\widehat{B}_{1}+\widehat{B}_{2} \subset |\Omega|.  \\
  \end{cases}
\end{equation*}
Since $\widehat{\Delta}_{1,3,5}\widehat{B}_{k}=\widehat{\Delta}_{2,4,6} \widehat{B}_{k}=0$ for $0\leq k\leq 3$ and $\widehat{\Delta}_{5,7,9}\widehat{B}_{k}=\widehat{\Delta}_{6,8,10} \widehat{B}_{k}=0$ for $3\leq k\leq 5$, $|\Omega|$ is base point free. Now by Bertini's theorem, we can prove the existence of surface in $\mathfrak{M}_{1,0}$ as before.

As a representative, we calculate the dimensions of $\mathcal{M}_{1,0}$. For a general surface in this family, we can assume that $\widehat{T}$ is the blowing-up of $\mathds{F}_{0}$ at 10 points $p_{i}(1\leq i\leq 10)$, where $p_{2k}$ and $p_{2k-1}$ are in the same fibre $B_{k}$ for $1\leq k\leq 5$. Since $2\overline{\delta}\equiv 6\Delta+5\Gamma-3\Sigma_{i=1}^{10} E_{i}$, we have $h^{0}(P,2\widehat{\delta})=h^{0}(P,2\overline{\delta})=24$. Since these 10 points in $\mathds{F}_{0}$ depend on $(2+1)\cdot 5=15$ parameters, the dimension of this family is $\mathrm{dim} \mathcal{M}_{1,0}=\mathrm{dim}|2\widehat{\delta}|+15-\mathrm{dim}(\mathrm{Aut}(\mathds{F}_{0}))=32$.
\end{proof}

In fact, all the 12 families in Theorem 6.4 form an irreducible component of the moduli space.  But at this section, we only prove that they are contained in the same irreducible component of the moduli space. The other part will be solved in Theorem \ref{component}. Here we recall the notations of Catanese and Pignatelli in \cite{Catanese} and the methods of Bauer and Pignatelli in \cite{Bauer 2}.

\begin{theorem}
The general surface in each of the 12 families in Theorem 6.4 admits a small deformation to a surface belonging to $\mathcal{M}_{1,0}$. In particular, all the surfaces in these 12 families are contained in the same irreducible component of the moduli space $\mathcal{M}_{\chi=5, K^{2}=9}$.
\end{theorem}
\begin{proof}

Each of the twelve families is contained in an irreducible component of the subscheme of the moduli space given by the surfaces having a canonical involution. We claim that the dimension of it is at least 32. Since for a small deformation in the subscheme preserving the number of the isolated fixed points of the involution, where the number is 5, then the twelve families are in the closure of the family $\mathcal{M}_{1,0}$.

\bigskip

Now we prove the claim. First, we use some of the techniques and results developed in \cite{Catanese}, which we will report in the case of a genus 2 fibration
$f:S\rightarrow \mathds{P}^{1}$. We can associate to the genus 2 fibration $f:S\rightarrow \mathds{P}^{1}$ the elements $V_{1},t,\varepsilon,\omega$, where
\begin{enumerate}[( 1 )]
\item $V_{1}=f_{*}w_{S|\mathds{P}^{1}}$ and $\mathrm{deg} V_{1}=\chi(\mathcal{O}_{S})+1=6$.
\item $t$ is an effective divisor of degree $K^{2}_{S}-2 \mathrm{deg} V_{1}+8=5$.
\item $\varepsilon$ is an element of $\mathrm{Ext}^{1}_{\mathcal{O}_{\mathds{P}^{1}}}(\mathcal{O}_{t},\mathrm{Sym}^{2}V_{1})/ \mathrm{Aut}_{\mathcal{O}_{\mathds{P}^{1}}}(\mathcal{O}_{t})$ giving the short exact sequence
    $$0 \rightarrow \mathrm{Sym}^{2}V_{1} \rightarrow f_{*}w_{S|\mathds{P}^{1}}^{2} \rightarrow \mathcal{O}_{t} \rightarrow 0$$
    where $\sigma_{2}:\mathrm{Sym}^{2}V_{1} \rightarrow V_{2}\doteq f_{*}w_{S|\mathds{P}^{1}}^{2}$ is the natural map induced by the tensor product of canonical sections of the fibres of $f$; $\sigma_{2}$ yields a rational map $\nu :\mathds{P}(V_{1})\dashrightarrow \mathds{P}(V_{2})$ (the relative version of 2-Veronese embedding $\mathds{P}^{1}\hookrightarrow \mathds{P}^{2}$) birational onto a conic bundle $\mathcal{C} \in |\mathcal{O}_{\mathds{P}(V_{2})}(2)\otimes \pi_{2}^{*}(\mathrm{det}V_{1})^{-2}|$ in $\mathds{P}(V_{2})$.
\item $\omega \in \mathds{P}(H^{0}(\mathds{P}^{1},\mathcal{A}_{6}(-22)))$ (Since $\mathrm{deg}(V_{3}^{+})^{2}=2(\mathrm{deg}V_{1}+\mathrm{deg}t)=22$), where $\mathcal{A}_{6}$ is a vector bundle obtained as a quotient of $\mathrm{Sym}^{3} f_{*}w_{S|\mathds{P}^{1}}^{2}$, the vector bundle of relative cubics on $\mathds{P}(f_{*}w_{S|\mathds{P}^{1}}^{2})$, by the subbundle of cubics vanishing on $\mathcal{C}$
    $$0\rightarrow f_{*}w_{S|\mathds{P}^{1}}^{2}\otimes \mathcal{O}_{\mathds{P}^{1}}(12)\rightarrow \mathrm{Sym}^{3} f_{*}w_{S|\mathds{P}^{1}}^{2} \rightarrow \mathcal{A}_{6} \rightarrow 0$$
    The branch curve $\mathcal{B}=\Delta \cap \mathcal{C}$ of the map $S \rightarrow \mathcal{C}$ is given by a map $\delta:\mathcal{O}_{\mathds{P}^{1}}(22) \rightarrow \mathcal{A}_{6}$, where $\Delta \in  |\mathcal{O}_{\mathds{P}(V_{2})}(3)\otimes \pi_{2}^{*}(V_{3}^{+})^{-2}|$ in $\mathds{P}(V_{2})$.
\end{enumerate}

\bigskip

The following proof is similar as that of Theorem 5.9  in \cite{Bauer 2}.

For the general surface in each of the twelve families, $\mathcal{C}$ has $\mathrm{deg} t = 5$ nodes (the vertices of the singular conics), none of them in $\mathcal{B}$, which is smooth. Let $\widehat{\mathcal{C}}$ be a
minimal desingularization of $\mathcal{C}$ and $\alpha: \widehat{\mathcal{C}}\rightarrow \mathcal{C}$; the 5 (-2)-curves $L_{i}$ on $\widehat{\mathcal{C}}$ give rise to 5 (-1)-curves on the associated double cover $\widehat{S}$, the exceptional locus of the birational morphism $\widehat{S} \rightarrow S$. The finite double cover $\varphi: \widehat{S} \rightarrow \widehat{\mathcal{C}}$ branches in $\widehat{\mathcal{B}}$, union of the pull-back of $\mathcal{B}$ with the (-2)-curves.

From the Theorem 2.16 in \cite{Catanese 1}, we have $\varphi_{*}(\Omega_{\widehat{S}}^{1}\otimes \Omega_{\widehat{S}}^{2})\backsimeq (\Omega_{\widehat{\mathcal{C}}}^{1}(log \widehat{\mathcal{B}})\otimes \Omega_{\widehat{\mathcal{C}}}^{2}) \oplus (\Omega_{\widehat{\mathcal{C}}}^{1}\otimes \Omega_{\widehat{\mathcal{C}}}^{2}(\frac{1}{2}\widehat{\mathcal{B}}))$ and $\Omega_{\widehat{\mathcal{C}}}^{1}(log \widehat{\mathcal{B}})\otimes \Omega_{\widehat{\mathcal{C}}}^{2}$ is the invariant part.

The morphism $\beta:\widehat{\mathcal{C}}\rightarrow \mathds{P}(V_{1})$ is the contraction of the strict transforms of each component of the singular conics, so of $2 \mathrm{deg} t = 10$ exceptional curves $E_{i}$ of the first kind. If $\mathcal{T}$ denotes the tangent sheaf, then $\chi(\mathcal{T}_{\mathds{P}(V_{1})})=\mathrm{deg}(ch(\mathcal{T}_{\mathds{P}(V_{1})})\cdot td(\mathcal{T}_{\mathds{P}(V_{1})}))_{2}=2K^{2}_{\mathds{P}(V_{1})}- 10\chi(\mathcal{O}_{\mathds{P}(V_{1})})=6$ by Hirzebruch-Riemann-Roch theorem. Besides, we have
$\mathcal{T}_{\widehat{\mathcal{C}}/ \mathds{F}_{0}} \simeq \oplus_{i=1}^{10}\mathcal{O}_{E_{i}}(1)$ by Lemma 22 in \cite{Horikawa 1} and the following exact sequence
$$0\rightarrow \mathcal{T}_{\widehat{\mathcal{C}}} \rightarrow \beta^{*}\mathcal{T}_{\mathds{F}_{0}} \rightarrow \mathcal{T}_{\widehat{\mathcal{C}}/ \mathds{F}_{0}}\rightarrow 0,$$
then $\chi(\mathcal{T}_{\widehat{\mathcal{C}}})=\chi(\mathcal{T}_{\mathds{P}(V_{1})})-4\mathrm{deg}t=6-20=-14$.

Let $H$ be the tautological bundle of $\mathds{P}(V_{2})$ and $\Gamma$ be the fibre of $\pi_{2}:\mathds{P}(V_{2})\rightarrow \mathds{P}^{1}$, we have the numerical equivalence $\mathcal{C} \sim 2H-12\Gamma$, $\Delta \sim 3H-22\Gamma$ and $k_{\mathds{P}(V_{2})} \sim -3H+21\Gamma$. Moreover, $H\Gamma^{2}=\Gamma^{3}=0$, $H^{2}\Gamma=1$ and $H^{3}-23H^{2}\Gamma=0$ by the Theorem 5.1 in \cite{BPV}, so we can calculate $\mathcal{B}(\mathcal{B}-k_{\mathcal{C}})=\Delta [\Delta-(\mathcal{C}+k_{\mathds{P}(V_{2})})]\mathcal{C}=(3H-22\Gamma)(4H-31\Gamma)(2H-12\Gamma)=46$, and

\begin{tabular}{m{2cm} m{8cm}}
& $h^{1}(\Omega_{\widehat{\mathcal{C}}}^{1}(log \widehat{\mathcal{B}})\otimes \Omega_{\widehat{\mathcal{C}}}^{2})- h^{2}(\Omega_{\widehat{\mathcal{C}}}^{1}(log \widehat{\mathcal{B}})\otimes \Omega_{\widehat{\mathcal{C}}}^{2})$ \\
& $\geq -\chi(\Omega_{\widehat{\mathcal{C}}}^{1}(log \widehat{\mathcal{B}})\otimes \Omega_{\widehat{\mathcal{C}}}^{2})=-\chi(\Omega_{\widehat{\mathcal{C}}}^{1}\otimes \Omega_{\widehat{\mathcal{C}}}^{2}) -\chi(\mathcal{O}_{\widehat{\mathcal{B}}}(\Omega_{\widehat{\mathcal{C}}}^{2}))$  \\
& $=-\chi(\mathcal{T}_{\widehat{\mathcal{C}}}) -\chi(\Omega_{\widehat{\mathcal{C}}}^{2}) +\chi(\Omega_{\widehat{\mathcal{C}}}^{2}(-\widehat{\mathcal{B}}))
=14+\frac{1}{2}(k_{\widehat{\mathcal{C}}}-\widehat{\mathcal{B}})(-\widehat{\mathcal{B}})$ \\
& $=9+\frac{1}{2}(\mathcal{B})(\mathcal{B}-k_{\mathcal{C}})=32$
\end{tabular}

\end{proof}

\begin{remark}
For a general surface $S$ in $\mathcal{M}_{1,0}$, we have $h^{2}(\Omega_{\widehat{\mathcal{C}}}^{1}(log \widehat{\mathcal{B}})\otimes \Omega_{\widehat{\mathcal{C}}}^{2})=h^{0}(S,\mathcal{T}_{S})=0$ since $S$ is of general type. By Theorem 6.4 and Theorem 6.5, we can get
$$32=\mathrm{dim} \mathcal{M}_{1,0}
\geq h^{1}(\Omega_{\widehat{\mathcal{C}}}^{1}(log \widehat{\mathcal{B}})\otimes \Omega_{\widehat{\mathcal{C}}}^{2})- h^{2}(\Omega_{\widehat{\mathcal{C}}}^{1}(log \widehat{\mathcal{B}})\otimes \Omega_{\widehat{\mathcal{C}}}^{2})
\geq -\chi(\Omega_{\widehat{\mathcal{C}}}^{1}(log \widehat{\mathcal{B}})\otimes \Omega_{\widehat{\mathcal{C}}}^{2})=32,$$
which means $h^{0}(\Omega_{\widehat{\mathcal{C}}}^{1}(log \widehat{\mathcal{B}})\otimes \Omega_{\widehat{\mathcal{C}}}^{2})=0$ and $h^{1}(\Omega_{\widehat{\mathcal{C}}}^{1}(log \widehat{\mathcal{B}})\otimes \Omega_{\widehat{\mathcal{C}}}^{2})=32$.
\end{remark}

\bigskip

\section{the moduli}

In the above sections, we classify all pairs $(S,i)$, where $S$ is a minimal surface with $K_{S}^{2}=9,\chi(S)=5$ and $i$ is a canonical involution on $S$, finding 6 families.

\bigskip

\begin{tabular}{m{1cm} m{1.5cm} m{9.7cm}}
\hline
Family & Theorem & Short Description \\
\hline
$\mathcal{M}_{1}$ & \ref{family 1}  & double cover of $\mathds{F}_{2}$ branched along the section $\Delta_{\infty}$ and a curve $B$, where $B$ is in $|9\Delta_{\infty}+18\Gamma|$ with eight ordinary 4-tuple points \\
\hline
\end{tabular}

\begin{tabular}{m{1cm} m{1.5cm} m{9.7cm}}
\hline
$\mathcal{M}_{2}$ & \ref{family 2}  & double cover of $\mathds{F}_{r}(0\leq r \leq 3)$ branched along a curve $B$ in $|8\Delta_{\infty}+(10+4r)\Gamma|$ with  seven ordinary 4-tuple points and one (3,3)-point \\
\hline
$\mathcal{M}_{3}$ & \ref{family 3}  &  double cover of weak Del Pezzo surface of degree 2 branched along a curve $B$ and three (-2)-curves $A_{i}$, where $B$ is in $|-5K|$ and $BA_{i}=A_{i}A_{j}=0$ \\
\hline
$\mathcal{M}_{4}$ & \ref{family 4}  & double cover of $\mathds{F}_{r}(0\leq r \leq 2)$ branched along a fibre $B_{1}$ and a curve $B_{2}$, where $B_{2}$ is in $|8\Delta_{\infty}+(9+4r)\Gamma|$ and has one (3,3)-point and five singular points $x_{1},\cdots,x_{5}$ of multiplicity 4, with $x_{4}\in B_{1}$ and $x_{5}$ infinitely near to $x_{4}$ belonging to the strict transform of $B_{1}$ \\
\hline
$\mathcal{M}_{5}$ & \ref{family 5}  & double cover of weak Del Pezzo surface of degree 6 branched along a curve in $|-4K|$ with three (3,3)-points \\
\hline
$\mathcal{M}_{6}$ & \ref{family 6}  & double cover of $\mathds{F}_{r}(r=0,2)$ branched along a curve in $|6\Delta+10\Gamma|$ or $|6\Delta+8\Gamma|$ or $|6\Delta+6\Gamma|$, having the possible singularities in Table \\
\hline
\end{tabular}

\bigskip

\begin{pro} \label{dimension}
The 6 families $\mathcal{M}_{i}(1\leq i\leq 6)$ is connected and unirational of respective dimensions 28, 27, 33, 32, 31 and 32.
\end{pro}
\begin{proof}
We have explained that each family is connected and unirational separately in the above sections. Now we calculate the dimension of each family.

For a general surface in $\mathcal{M}_{1}$, we can assume that $P$ is the blowing-up of $\mathds{F}_{2}$ at 8 points $p_{i}(1\leq i\leq8)$. By Theorem \ref{family 1}  we have $h^{0}(P,2\widehat{\delta})=20$. Since these 8 points in $\mathds{F}_{2}$ depend on 16 parameters, the dimension of this family is $\mathrm{dim} \mathcal{M}_{1}=(20-1)+16-7=28$.

For a general surface in $\mathcal{M}_{2}$, we can assume that $P$ is the blowing-up of $\mathds{F}_{0}$ at 9 points $p_{i}(1\leq i\leq9)$, and $p_{1}$ is in the preimage of $p_{2}$. By Theorem \ref{family 2} we have $h^{0}(P,2\widehat{\delta})=17$. Since these 9 points in $\mathds{F}_{0}$ depend on 17 parameters, the dimension of this family is $\mathrm{dim}\mathcal{M}_{2}=(17-1)+17-6=27$.

For a general surface in $\mathcal{M}_{3}$, we can assume that $P$ is the blowing-up of $\mathds{P}^{2}$ at 7 points $p_{i}(1\leq i\leq 7)$, and $p_{2k-1}$ is in the preimage of $p_{2k}$ where $k=1,2,3$. By Theorem \ref{family 3} we have $h^{0}(P,2\widehat{\delta})=31$. Since these 7 points in $\mathds{P}^{2}$ depend on 11 parameters, the dimension of this family is $\mathrm{dim} \mathcal{M}_{3}=(31-1)+11-8=33$.

For a general surface in $\mathcal{M}_{4}$, we can assume that $P$ is the blowing-up of $\mathds{F}_{0}$ at 7 points $p_{i}(1\leq i \leq 7)$.  $p_{1}$ is in the preimage of $p_{2}$, while $p_{3}$ is the intersection of the preimage of $p_{4}$ and the strict transform of the fibre $B_{1}$ passing through $p_{4}$. By Theorem \ref{family 4} we have $h^{0}(P,2\widehat{\delta})=28$. Since these 7 points in $\mathds{F}_{0}$ depend on 11 parameters, the dimension of this family is $\mathrm{dim} \mathcal{M}_{4}=(28-1)+11-6=32$.

For a general surface in $\mathcal{M}_{5}$, we can assume that P is the blowing-up of $\mathds{P}^{2}$ at 9 points $p_{i}(1\leq i\leq 9)$, and $p_{2k-1}$ is in the preimage of $p_{2k}$ where $k=1,2,3$. By Theorem \ref{family 5} we have $h^{0}(P,2\widehat{\delta})=25$. Since these 9 points in $\mathds{P}^{2}$ depend on 15 parameters, the dimension of this family is $\mathrm{dim}\mathcal{M}_{5}=(25-1)+15-8=31$.

For the family $\mathcal{M}_{6}$, we have done in Theorem \ref{family 6}.
\end{proof}

\begin{remark}
By Kuranishi's theorem each irreducible component of the moduli space of minimal surfaces of general type with $K_{S}^{2}=9,\chi(S)=5$ has dimension at least $10\chi -2 K^{2}= 32$. It follows that the three families $\mathcal{M}_{1}$, $\mathcal{M}_{2}$ and $\mathcal{M}_{5}$ are not irreducible components of the moduli space $\mathcal{M}_{\chi=5, K^{2}=9}$. Observe that the general point of the irreducible component, in which the family $\mathfrak{M}_{1}$ or $\mathcal{M}_{2}$ is contained, is a surface without a canonical involution. In fact, it cannot be in $\mathcal{M}_{i}(3\leq i \leq 5)$ because $\tau$ is invariant under deformations preserving the involution.
\end{remark}

In the following theorem, we prove that the family of surfaces with a genus 2 fibration forms an irreducible component of the moduli space.

\begin{theorem}\label{component}
The family $\mathcal{M}_{6}$ is an irreducible components of the moduli space $\mathcal{M}_{\chi=5, K^{2}=9}$.
\end{theorem}
\begin{proof}
By Theorem \ref{dimension}, $\mathcal{M}_{6}$ is of dimension 32. To prove that it is an irreducible component of the moduli space $\mathcal{M}_{\chi=5, K^{2}=9}$, we only need to show that any general surface $S$ in this family satisfies $h^{1}(S,\mathcal{T}_{S})=32$.

For the general surface $S$ in $\mathcal{M}_{6}$, we have the construction described in Theorem 6.5, which is the following commutative graph
$$\xymatrix{
    \widehat{S} \ar[r]^{\varphi}\ar[d]_{\epsilon} & \widehat{\mathcal{C}} \ar[d]_{\beta} \\
    S \ar@{-->}[r]^{} &  \mathds{F}_{0}
    }$$
where $\epsilon$ is the blowing-up at 5 points, $\varphi$ is a finite double cover, and $\beta$ is the blowing-up at 10 points. Besides, let $x_{i}, y_{i}(1\leq i \leq 5)$ be the 10 points on $\mathds{F}_{0}$, and $E_{x_{i}},E_{y_{i}}(1\leq i \leq 5)$ be the corresponding (-1)-curves on $\widehat{\mathcal{C}}$, then the branch curve $\widehat{\mathcal{B}}$ of the finite double cover $\varphi$ is in $|6\Gamma_{1}+16\Gamma_{2}-4\Sigma_{i=1}^{5}(E_{x_{i}}+E_{y_{i}})|$. Actually, we have $\widehat{\mathcal{B}}\equiv \widehat{\mathcal{B}}_{1}+\Sigma_{i=1}^{5}L_{i}$, where $L_{i}\equiv \Gamma_{2}-E_{x_{i}}-E_{y_{i}}$ is a base part of $|\widehat{\mathcal{B}}|$ and $\widehat{\mathcal{B}}_{1} \equiv 6\Gamma_{1}+11\Gamma_{2}-3\Sigma_{i=1}^{5}(E_{x_{i}}+E_{y_{i}})$.

We calculate the dimension of $H^{i}(\widehat{S},\mathcal{T}_{\widehat{S}})$ first. Since $h^{i}(\widehat{S},\mathcal{T}_{\widehat{S}})=h^{2-i}(\widehat{S},\Omega_{\widehat{S}}^{1}\otimes \Omega_{\widehat{S}}^{2})$ by Serre duality and $\varphi_{*}(\Omega_{\widehat{S}}^{1}\otimes \Omega_{\widehat{S}}^{2})\simeq
(\Omega_{\widehat{\mathcal{C}}}^{1}(log \widehat{\mathcal{B}})\otimes \Omega_{\widehat{\mathcal{C}}}^{2}) \oplus
(\Omega_{\widehat{\mathcal{C}}}^{1}\otimes \Omega_{\widehat{\mathcal{C}}}^{2}(L))$ by Theorem 2.16 in \cite{Catanese 1}, where we denote $2L\equiv \widehat{\mathcal{B}}$, then we can get $h^{i}(\widehat{S},\mathcal{T}_{\widehat{S}})= h^{2-i}(\widehat{\mathcal{C}},\Omega_{\widehat{\mathcal{C}}}^{1}(log \widehat{\mathcal{B}})\otimes \Omega_{\widehat{\mathcal{C}}}^{2})+ h^{2-i}(\widehat{\mathcal{C}},\Omega_{\widehat{\mathcal{C}}}^{1}\otimes \Omega_{\widehat{\mathcal{C}}}^{2}(L))$ for $i=0,1,2$.

For the invariant part $\Omega_{\widehat{\mathcal{C}}}^{1}(log \widehat{\mathcal{B}}) \otimes \Omega_{\widehat{\mathcal{C}}}^{2}$, we have $h^{0}(\Omega_{\widehat{\mathcal{C}}}^{1}(log \widehat{\mathcal{B}}) \otimes \Omega_{\widehat{\mathcal{C}}}^{2})= h^{2}(\Omega_{\widehat{\mathcal{C}}}^{1}(log \widehat{\mathcal{B}}) \otimes \Omega_{\widehat{\mathcal{C}}}^{2})=0$ and $h^{2}(\Omega_{\widehat{\mathcal{C}}}^{1}(log \widehat{\mathcal{B}}) \otimes \Omega_{\widehat{\mathcal{C}}}^{2})=32$ by Remark 6.6. For the anti-invariant part $\Omega_{\widehat{\mathcal{C}}}^{1}\otimes \Omega_{\widehat{\mathcal{C}}}^{2}(L)$, we have $h^{i}(\widehat{\mathcal{C}}, \Omega_{\widehat{\mathcal{C}}}^{1}\otimes \Omega_{\widehat{\mathcal{C}}}^{2}(L)) =
h^{2-i}(\widehat{\mathcal{C}}, \mathcal{T}_{\widehat{\mathcal{C}}}(-L))$ by Serre duality and the exact sequence
$$0\rightarrow \mathcal{T}_{\widehat{\mathcal{C}}}(-L) \rightarrow \beta^{*}\mathcal{T}_{\mathds{F}_{0}}(-L) \rightarrow \mathcal{T}_{\widehat{\mathcal{C}}/ \mathds{F}_{0}}(-L)\rightarrow 0.$$
Since $\mathcal{T}_{\widehat{\mathcal{C}}/ \mathds{F}_{0}} \simeq \oplus_{i=1}^{5}(\mathcal{O}_{E_{x_{i}}}(1)\oplus \mathcal{O}_{E_{y_{i}}}(1))$ by Lemma 22 in \cite{Horikawa 1} and $E_{x_{i}}L_{i}=E_{y_{i}}L_{i}=2$, we have  $h^{i}(\mathcal{T}_{\widehat{\mathcal{C}}/ \mathds{F}_{0}} (-L))=0$ for any $i$. Since $\beta^{*}\mathcal{T}_{\mathds{F}_{0}}(-L)\simeq \mathcal{O}(-\Gamma_{1}-8\Gamma_{2}+2\Sigma_{i=1}^{5}(E_{x_{i}}+E_{y_{i}}))\oplus \mathcal{O}(-3\Gamma_{1}-6\Gamma_{2}+2\Sigma_{i=1}^{5}(E_{x_{i}}+E_{y_{i}}))$, we can calculate
\begin{equation*}
  \begin{cases}
    h^{i}(\mathcal{O}(-\Gamma_{1}-8\Gamma_{2}+2\Sigma_{i=1}^{5}(E_{x_{i}}+E_{y_{i}})))=0, i=0,2 \\
    h^{1}(\mathcal{O}(-\Gamma_{1}-8\Gamma_{2}+2\Sigma_{i=1}^{5}(E_{x_{i}}+E_{y_{i}})))=10, \\
    h^{i}(\mathcal{O}(-3\Gamma_{1}-6\Gamma_{2}+2\Sigma_{i=1}^{5}(E_{x_{i}}+E_{y_{i}})))=0, i=0,1,2. \\
  \end{cases}
\end{equation*}
Thus, we have $h^{i}(\widehat{\mathcal{C}}, \Omega_{\widehat{\mathcal{C}}}^{1}\otimes \Omega_{\widehat{\mathcal{C}}}^{2}(L))=0$ for $i=0,2$ and $h^{1}(\widehat{\mathcal{C}}, \Omega_{\widehat{\mathcal{C}}}^{1}\otimes \Omega_{\widehat{\mathcal{C}}}^{2}(L)) =10$. In conclusion, we can get $h^{0}(\widehat{S},\mathcal{T}_{\widehat{S}})=h^{2}(\widehat{S},\mathcal{T}_{\widehat{S}})=0$ and $h^{1}(\widehat{S},\mathcal{T}_{\widehat{S}})=32+10=42$.

Now we calculate the dimension of $H^{i}(S,\mathcal{T}_{S})$ using the following exact sequence
$$ 0\rightarrow \mathcal{T}_{\widehat{S}} \rightarrow \epsilon^{*}\mathcal{T}_{S} \rightarrow \mathcal{T}_{\widehat{S}/S} \rightarrow 0 .$$
We remark that $h^{0}(S,\mathcal{T}_{S})=0$ and $h^{i}(\widehat{S},\epsilon^{*}\mathcal{T}_{S})=h^{i}(S,\mathcal{T}_{S})$ for any $i$. Since $\epsilon$ is the blowing-up at 5 points, then we get $h^{0}(\mathcal{T}_{\widehat{S}/S})=10$ and $h^{1}(\mathcal{T}_{\widehat{S}/S})=0$ by Lemma 22 in \cite{Horikawa 1}. Hence, we have $h^{2}(S,\mathcal{T}_{S})=0$, $h^{2}(\widehat{S},\mathcal{T}_{\widehat{S}})=0$ and $h^{1}(S,\mathcal{T}_{S})=32$.
\end{proof}

\bigskip

\section*{Appendix}

Here, we give the algorithm to judge whether an effective divisor on $\mathds{F}_{0}$  is nef or not. In this process, we only need to solve a system of linear equations and find the irreducible decompostion of any given effective divisor. It's similar for the divisor on $\mathds{P}^{2}$.

As a representative, we consider the example in Theorem \ref{family 2}. Let $\Delta_{\infty}$ be the section at infinity and $\Gamma$ be a fibre of the Hirzebruch surface $\mathds{F}_{0}$. We assume that $P$ is the blowing-up of $\mathds{F}_{0}$ at 9 points $p_{i}(1\leq i\leq9)$, where $p_{1}$ is in the preimage of $p_{2}$. If we denote $f:P\rightarrow \mathds{F}_{0}$ and $L\doteq 10\Delta_{\infty}+12\Gamma-5\Sigma_{i=1}^{9}E_{i}$ where $E_{i}=f^{-1}(p_{i})$ for $1\leq i\leq9$, we need to prove that $L$ is nef.

First, we give a decompostion $L=L_{1}+L_{2}+L_{3}$, where
\begin{equation*}
  \begin{cases}
    L_{1}\doteq 2\Delta_{\infty}+2\Gamma-\Sigma_{i=1}^{7}E_{i}-E_{8}, \\
    L_{2}\doteq 2\Delta_{\infty}+2\Gamma-\Sigma_{i=1}^{7}E_{i}-E_{9}, \\
    L_{3}\doteq 6\Delta_{\infty}+8\Gamma-3\Sigma_{i=1}^{7}E_{i}-4\Sigma_{i=8}^{9}E_{i}. \\
  \end{cases}
\end{equation*}
Now $L_{1}$ and $L_{2}$ are nef by Lemma \ref{nef 1} and $L.L_{1}=L.L_{2}=4$. Since $h^{0}(P,L_{3})\geq 1$ and $L.L_{3}=7$, we only need to prove that $L_{3}$ is effective.
So it is enough to prove that $L_{2}$ is irreducible.

Secondly, we need to give a defining equation of the image of $L_{3}$ on $\mathds{F}_{0}$. We can choose a coordinate system such that the defining equation is
$$g(x,y)\doteq \sum_{0\leq i \leq 6, 0\leq j \leq 8}a_{i,j}x^{i}y^{j},$$
where $div(x)\equiv\Delta_{\infty}$ and $div(y)\equiv\Gamma$ and $\{a_{i,j}\}$ are undetermined coefficients. We denote $p_{k}\doteq (x_{k},y_{k})$ and
$$g_{\alpha,\beta}(p_{k})\doteq \frac{\partial^{\alpha+\beta}}{\partial x^{\alpha}\partial y^{\beta}}g(x_{k},y_{k}).$$
Since $p_{8}$ and $p_{9}$ are ordinary 4 points of $g(x,y)$, so we have the following 20 linear equations of $\{a_{i,j}\}$
$$g_{\alpha,\beta}(p_{k})=0(0\leq \alpha,0\leq \beta,0\leq \alpha+\beta \leq 3, 8\leq k\leq 9).$$
Since $p_{k}(3\leq k\leq 7)$ are ordinary 3 points of $g(x,y)$, so we have the following 30 linear equations of $\{a_{i,j}\}$
$$g_{\alpha,\beta}(p_{k})=0(0\leq \alpha,0\leq \beta,0\leq \alpha+\beta \leq 2, 3\leq k\leq 7).$$
Let $C$ be the section passing through $p_{2}$. For $p_{1}$ and $p_{2}$, we can assume that $p_{1}$ is the intersection of $E_{2}$ and the strict transform of $C$ to reduce the difficulty of calculation. So we have the following 12 linear equations of $\{a_{i,j}\}$
$$g_{\alpha,\beta}(p_{k})=0(0\leq \alpha,0\leq \beta ,0\leq \alpha+2\beta \leq 2+3, k=2).$$
Now we have 63 undetermined coefficients $\{a_{i,j}\}$ and 62 homogeneous linear equations. If we assume that
$$a_{0,0}=1,$$
then we can get a unique solution $\{a_{i,j}\}$ and a unique defining equation $g(x,y)$, with the help of computer, for the general choice of $\{p_{k}\}$.

Finally, we can use the function command ``factor()'' in MATLAB to give the irreducible decompostion of $g(x,y)$. After the calculation, we know that $g(x,y)$ is irreducible for the general choice of $\{p_{k}\}$.

\end{document}